\documentclass{sapm}

\usepackage{graphicx,amsmath,amssymb,geometry,subcaption}
\usepackage[usenames,dvipsnames]{xcolor}
\usepackage{bbm,listings}
\lstdefinelanguage{Julia}%
  {morekeywords={abstract,break,case,catch,const,continue,do,else,elseif,%
      end,export,false,for,function,immutable,import,importall,if,julia,in,%
      macro,module,otherwise,quote,return,switch,true,try,type,typealias,%
      using,while},%
   sensitive=true,%
   alsoother={},%
   morecomment=[l]\#,%
   morecomment=[n]{\#=}{=\#},%
   morestring=[s]{"}{"},%
   morestring=[m]{'}{'},%
}[keywords,comments,strings]%

\newcommand{\amp}{A}

\newcommand{\tcr}[1]{{#1}}
\usepackage{cite}

\newcommand{\D}{\mathrm{d}}
\newcommand{\E}{\mathrm{e}}
\newcommand{\I}{\mathrm{i}}
\DeclareMathOperator{\imag}{Im}
\DeclareMathOperator{\real}{Re}
\DeclareMathOperator{\sign}{sign}

\def\XXint#1#2#3{{\setbox0=\hbox{$#1{#2#3}{\int}$}
     \vcenter{\hbox{$#2#3$}}\kern-.5\wd0}}

\usepackage[labelformat=simple]{subcaption}

\newcommand*{\Scale}[2][4]{\scalebox{#1}{\ensuremath{#2}}}
\newtheorem{RHproblem}{Riemann--Hilbert Problem}

\renewcommand{\cdot}{\diamond}
\usepackage{hyperref}
\begin{document}

\title[Scattering and inverse scattering for the AKNS system]{Scattering and inverse scattering for the AKNS system:  A rational function approach}  

\author[T. Trogdon]{Thomas Trogdon\thanks{Address for
correspondence: Thomas Trogdon, Department of Applied Mathematics,
University of Washington, WA 98195-3925; email:
trogdon@uw.edu}}\affil{Department of Applied Mathematics,
University of Washington, WA, USA, 98195-3925}

\maketitle

\begin{center}
  \emph{ \small This paper is based on a talk given at the Recent Advances in Nonlinear Waves conference at the University of Washington in 2017 in honor of Harvey Segur's 75th birthday.}
\end{center}
\begin{abstract}
We consider the use of rational basis functions to compute the scattering and inverse scattering transforms associated with the AKNS system. The proposed numerical forward scattering transform computes the solution of the AKNS system that is valid on the entire real axis and thereby computes a reflection coefficient at a point by solving a single linear system. The proposed numerical inverse scattering transform makes use of a novel improvement in the rational function approach to the oscillatory Cauchy operator, enabling the efficient solution of certain Riemann--Hilbert problems without contour deformations.  The latter development enables access to high-precision computations and this is demonstrated on the inverse scattering transform for the one-dimensional Schr\"odinger operator with a $\mathrm{sech}^2$ potential.
\end{abstract}

\section{Introduction}

We consider scattering and inverse scattering for the (Ablowitz-Kaup-Newell-Segur) AKNS \cite{Ablowitz1974} system
\begin{align}\label{eq:AKNS}
\frac{\D\mu }{\D x}(x;k) = \begin{bmatrix} - \I k & q(x) \\ r(x) & \I k \end{bmatrix} \mu(x;k), \quad \mu(\cdot;k): \mathbb R \to \mathbb C^{2 \times 2}. 
\end{align}
This includes the system considered previously by Zakharov and Shabat \cite{Zakharov-Shabat} by restricting to the case where $r(x) = - \overline{q(x)}$ and that considered by Gardner-Green-Kruskal-Miura \cite{GGKM} by setting $q = -1$.  The theoretical developments related to the scattering transforms for this system are reviewed in Section~\ref{sec:scattering} and are entirely classical.  Alternate references include the original AKNS paper \cite{Ablowitz1974}, the paper by Deift and Trubowitz \cite{Deift1979} and the books \cite{AblowitzClarksonSolitons,FokasComplexVariables,johnson,Ontheline}, to name a few.  To avoid technical considerations, we will suppose that $q,r$ are infinitely differentiable functions with some degree of exponential decay.  On occasion we will point to weaker assumptions for different aspects of the theory.   

An overview of the approach is that \eqref{eq:AKNS} represents a spectral problem with spectral parameter $k$.  Indeed, it is elementary to turn \eqref{eq:AKNS} into a \emph{bona fide} eigenvalue problem, see \eqref{eq:ode-evalprob}.  For each $k$ in the spectrum there is spectral data defined.  Because \eqref{eq:AKNS}, as a densely defined operator on $L^2(\mathbb R)$, is a relatively compact perturbation of a self-adjoint operator, the spectrum consists of $\mathbb R$ with the possible addition of a discrete set of $L^2(\mathbb R)$ eigenvalues \cite{Schechter1965}.  It is indeed a non-trivial statement that the potentials $q,r$ can be recovered from the spectral data.  The process of computing the spectral data is called scattering (or the scattering transform) and the process of reconstructing the potentials is called inverse scattering (or the inverse scattering transform).  These transforms are necessarily nonlinear.

A number of authors have developed numerical approaches to compute the scattering and inverse scattering transforms, see \cite{boffetta,Osborne1991,Wahls2018,TrogdonSONLS,TrogdonSOKdV}, for example.  Some, such as the approach in \cite{Wahls2018}, focus on speed while others, see \cite{TrogdonSONLS,TrogdonSONNSD}, for example, focus on accuracy. The reason for this tradeoff can be understood in the following way.  \tcr{When one uses the fast Fourier transform (FFT) to approximate true Fourier coefficients,} one obtains both speed and accuracy --- (1) the transform requires just $O(n \log n)$ floating point operations to compute $n$ Fourier coefficients and (2) for smooth functions these coefficients converge spectrally\footnote{We say that $f_n \to f$ spectrally if $|f_n - f| = O(n^{-\alpha})$ for any fixed $\alpha > 0$ as $n \to \infty$.}.  But this could be interpreted as a happy accident because equally-spaced nodes for periodic functions both enable the exponentially convergent trapezoidal rule \cite{Trefethen2014} and the butterfly decomposition for the discrete Fourier transform matrix.  As this structure is deformed, one may have to choose accuracy over speed or vice versa.  \tcr{This dichotomy has been disturbed by the recent work in \cite{Chimmalgi2019} where high-order integrators are used to obtain both speed and accuracy for the scattering transform.}

The problem at hand is more difficult due to nonlinearity --- in the classical Fourier setting a ten-fold increase in amplitude will have a marginal effect on overall accuracy (loss of one digit of absolute accuracy) yet for a scattering transform this should be expected to noticeably increase the amount of computational work required to achieve the same accuracy.  For example, \tcr{an increase in amplitude for $q,r$ by a factor of 10} could turn an operator, producing no $L^2(\mathbb R)$ eigenvalues, to one producing 62 eigenvalues.

The approach in the current paper is an entirely rational function-based approach to computing both the scattering and inverse scattering transforms.  The rational basis we consider is convenient because it is closed under differentiation, function multiplication and the action of the Cauchy operator.  In Section~\ref{sec:FT} we discuss the use of a rational basis to derive two methods to compute the classical Fourier transform on $\mathbb R$.  These methods, while both effective, serve to indicate a way forward to compute with the scattering transforms.  In Section~\ref{sec:scattering} we describe the classical scattering transforms for the AKNS system.  We determine both left and right scattering data.  In Section~\ref{sec:numerical_scattering} we describe how to extend the first method in Section~\ref{sec:FT} to compute the scattering data for the AKNS system.  In Section~\ref{sec:inverse} we describe how to compute the inverse scattering transform by solving Riemann--Hilbert problems numerically.  These Riemann--Hilbert problems are equivalent to singular integral equations and the (infinite-dimensional) GMRES algorithm can be applied as in \cite{Trogdon2013} using the so-called oscillatory Cauchy operator without any need for contour deformations.   Most of the previous work using Riemann--Hilbert theory, see \cite{TrogdonSONLS}, for example, had required contour deformations. 
The crucial improvement made in the current work is detailed in Appendix~\ref{sec:cauchy}:  The numerical method in \cite{Trogdon2013} required the unstable evaluation of special functions whereas the method developed in the current work avoids this instability by leveraging the FFT and increases the speed to $O(m^2)$  to compute the oscillatory Cauchy operator of a linear combination of $m$ basis functions --- potentially suboptimal but still a vast improvement.  Due to this improvement, high-precision computations can be made and in Section~\ref{sec:hi-prec} we compute the inverse scattering transform for the Schr\"odinger operator with a potential of the form $U_0 \mathrm{sech}^2(x)$ to nearly a uniform accuracy of $10^{-50}$.  

\tcr{The aforementioned deformations of the the associated singular integral operator and the discretization of it as discussed in \cite{SORHFramework} leads to a dense linear system. While this is sufficient for the computations in \cite{TrogdonSONLS}, and has increasing benefit for the large-time regime in the solution of the nonlinear Schr\"odinger equation, it precludes the use of higher-precision arithmetic.  For example, to produce the solution discussed in  Figure~\ref{f:hiprec-err}, without using an iterative method, one would need to consider a linear system of size approximately $8192 \times 8192$ with entries stored using $70$ digits of accuracy.  Supposing that such a number takes approximately $4 \times 64$ bits to store it, such a linear system would require over 4 gigabyes of memory to store, and a long time to construct.  This should be compared with storage of less than a megabyte per GMRES iteration for the method proposed here with only an FFT for precomputation.  And this efficiency enables trivial parallel computation. }

\begin{remark}
All code required to produce the plots in the paper is available, written, almost entirely, in the {\tt Julia} language  \cite{Trogdon2021,Trogdon2021a}.
\end{remark}
 
We end the introduction with the fixing of some elementary notation.  Let
\begin{align*}
    \mathbb C^\pm = \{z \in \mathbb C : \pm \real z > 0 \}.
\end{align*}
The character $\hat{\phantom{f}}$ will be used to denote the Fourier transform of a function
\begin{align*}
  \hat f(k) = \int_{-\infty}^\infty f(x) \E^{-\I k x} \D x.
\end{align*}
And $(\cdot)$ will be used to refer to the independent variable of a function, i.e., if $g(x) = x f(x)$ we write $g = (\cdot) f(\cdot)$.

\section{Two methods for computing the Fourier transform}\label{sec:FT}

The forward and inverse transforms that are computed in the following sections can best be understood, at a high level, by a discussion of the computation of the Fourier transform.  Consider the differential equation
\begin{align}\label{eq:scalarmu}
\frac{\D \mu}{\D x}(x;k) - \I k \mu(x;k) = q(x).
\end{align}
All solutions of this are of the form
\begin{align*}
    \mu(x;k) = c_0 \E^{\I k x} + \int_a^x \E^{\I k(x-s)} q(s) \D s.
\end{align*}
Next, we consider two solutions $\mu^\pm(x;k)$ satisfying $\lim_{x \to \pm\infty} \mu^\pm(x;k) = 0$. Thus
\begin{align*}
    \mu^\pm(x;k) = \int_{\pm \infty}^x\E^{\I k(x-s)} q(s) \D s.
\end{align*}
It follows immediately that $\mu^\mp(x;k)$, as a function of $k$, has an analytic extension  to ${\mathbb C^\pm}$ that extends continuously to $\mathbb R$. Define a sectionally analytic function $m: \mathbb C \setminus \mathbb R \to \mathbb C$
\begin{align*}
    m(k) = m(k;x) = \begin{cases} \mu^-(x;k) & \imag k > 0,\\
    \mu^+(x;k) & \imag k < 0. \end{cases}
\end{align*}
For each $x$, it follows that $m$ solves the following Riemann--Hilbert  problem, provided $q \in L^1 \cap L^2(\mathbb R)$.
\begin{RHproblem}\label{rhp:0}
Find $u(\cdot;x): \mathbb R \to \mathbb C$ such that
\begin{align*}
    m(k;x) = \frac{1}{2\pi \I} \int_{-\infty}^\infty \frac{u(k';x)}{k' -k} \D k', \quad k \not \in \mathbb R,
\end{align*}
and
\begin{align*}
    m^+(k) - m^-(k) &= \E^{\I k x} \hat q(k), \quad \hat q(k)=  \int_{-\infty}^\infty \E^{-\I ks} q(s) \D s,\\
    m^\pm(k) &= \lim_{\epsilon \to 0^+} m( k \pm \I \epsilon).
\end{align*}
\end{RHproblem}
We note two things:
\begin{enumerate}
    \item We can compute $\hat q(k)$ from the large $x$ behavior of $\mu^-$ using
    \begin{align}\label{eq:large-x-ft}
        \lim_{x \to \infty} \E^{- \I k x} \mu^-(x;k) = \hat q(k).
    \end{align}
\item From the differential equation \eqref{eq:scalarmu} it follows that (see \cite[Section 2.5.3]{TrogdonSOBook}, for example)
\begin{align*}
    -\lim_{|k| \to \infty} \I k m(x;k) = q(x).
\end{align*}
\end{enumerate}

We now use these two observations to compute the Fourier transform in two ways.

\subsection{The ODE method for the Fourier transform}\label{sec:ode}

We implement a numerical method to compute the Fourier transform via \eqref{eq:AKNS} using rational basis functions.  This method can be considered a Levin-type method \cite{Levin} (see also \cite{A.Iserles2006}).  Define, using the basis defined in \eqref{eq:Rjs},
\begin{align*}
    \mathfrak r(x) =  \frac{4\nu}{\nu^2+x^2} =  - R_{1,0}(x) - R_{-1,0}(x).
\end{align*}
For an integrable function $f$ define
\begin{align*}
    \phi(x;k,f) =\phi(x;s) = \int_{-\infty}^x \E^{\I k (x -s)} f(s) \D s.
\end{align*}
We seek a solution of \eqref{eq:scalarmu} of the form
\begin{align*}
    \mu(x;k) = c_0 \phi(x;k,\mathfrak r) + \sum_{j \neq 0} c_j R_{j,0}(x).
\end{align*}
A formula for $\phi(x;k,\mathfrak r)$, $\nu = 1$, can be found in terms of incomplete Gamma functions $\Gamma(a,z)$ \cite{DLMF}
\begin{align*}
\phi(x;k,\mathfrak r) &= \frac{1}{2 \pi \I} \begin{cases} U(x,k) & k \leq 0,\\
- \overline{U(x,-k)} & k > 0, \end{cases}\\
U(x,k) &= \E^{\I k x} \left( \E^{-k} \Gamma(0,\I k x - k) - \E^{k} \Gamma(0,\I k x + k)  + \pi \I \E^{k} (1 + \mathrm{sign}(x))\right),\\
\mathrm{sign}(x) &= \begin{cases} 1 & x \geq 0,\\ -1 & x < 0. \end{cases}
\end{align*}
This can be derived using Riemann--Hilbert Problem~\ref{rhp:0}.  But we do not need this formula in what follows:  We only need the limit \eqref{eq:large-x-ft}, not the evaluation of $\mu(x;k)$ pointwise.

The unknowns $c_0,c_{\pm 1}, c_{\pm 2},\ldots$ are ordered as
\begin{align*}
    \vec c = \begin{bmatrix}
    c_0 & c_1 & c_{-1} & c_2 & c_{-2} & \cdots & 
    \end{bmatrix}^T.
\end{align*}
Then the action of the differential operator $\frac{\D}{\D x}$ on the basis $\{R_{j,0}\}$ is captured by
\begin{align}\label{eq:Dnu}
   \mathcal D_\nu = \begin{bmatrix} \frac{\I}{\nu} & & -\frac{\I}{\nu} \\
   & -\frac{\I}{\nu} & & \frac{\I}{\nu} \\
    - \frac{ \I}{2\nu}  & & \frac{2\I}{\nu} &  & - \frac{3 \I}{2\nu}\\
     & \frac{ \I}{2\nu}  & & -\frac{2\I}{\nu} &  &  \frac{3 \I}{2\nu}\\
     &  & -\frac{ \I}{\nu}  & & \frac{3\I}{\nu} &  &  -\frac{2 \I}{\nu}\\
     & &  & \frac{ \I}{\nu}  & & -\frac{3\I}{\nu} &  &  \frac{2 \I}{\nu}\\
     & & &  & - \frac{3 \I}{2\nu}  & & \frac{4 \I}{\nu} &  &  \ddots \\
     &&&&&\ddots&& \ddots
   \end{bmatrix},
\end{align}
see \eqref{eq:tridiag-diff}.  Suppose
\begin{align*}
    \sum_{j \neq 0} q_j R_{j,0}(x) = q(x).
\end{align*}
Then we must solve the banded system
\begin{align}\label{eq:banded}
     \Scale[.9]{\begin{bmatrix} -1 & \frac{\I}{\nu} & & -\frac{\I}{\nu} \\
   -1 && -\frac{\I}{\nu} & & \frac{\I}{\nu} \\
    &- \frac{ \I}{2\nu}  & & \frac{2\I}{\nu} &  & - \frac{3 \I}{2\nu}\\
     && \frac{ \I}{2\nu}  & & -\frac{2\I}{\nu} &  &  \frac{3 \I}{2\nu}\\
     &&  & -\frac{ \I}{\nu}  & & \frac{3\I}{\nu} &  &  -\frac{2 \I}{\nu}\\
     && &  & \frac{ \I}{\nu}  & & -\frac{3\I}{\nu} &  &  \frac{2 \I}{\nu}\\
     && & &  & - \frac{3 \I}{2\nu}  & & \frac{4 \I}{\nu} &  &  \ddots \\
     &&&&&&\ddots&& \ddots
   \end{bmatrix}\begin{bmatrix}
    c_0 \\ c_1 \\ c_{-1} \\ c_2 \\ c_{-2} \\ \vdots 
    \end{bmatrix} = \begin{bmatrix}
     q_1 \\ q_{-1} \\ q_2 \\ q_{-2} \\ \vdots 
    \end{bmatrix}.}
\end{align}
This can be done using the adaptive QR algorithm \cite{Olver2013a}, for example, to a prescribed tolerance. \tcr{ Specifically, the adaptive QR algorithm is applicable to semi-infinite linear systems $A \vec x = \vec b$ that are banded below and it can be implemented using Givens rotations to put a matrix into upper-triangular form while monitoring the right-hand side vector throughout the process.  At a finite step $k$ one arrives at a semi-infinite linear system
\begin{align*}
    \begin{bmatrix}R_k & U_k \\ 0 & A_k \end{bmatrix}  \vec x = Q_k^T \vec b =  \begin{bmatrix} \vec b_k \\ \vec r_k \end{bmatrix}
\end{align*}
where $R_k$ is upper-triangular.  If $\|\vec r_k\|_2$ is small and the original operator is well conditioned, one is guaranteed that $\vec x_k := R_k^{-1}\vec b_k$, augmented with zeros, is close to $\vec x.$  And if $\vec b$ has entries that decay rapidly (which is the case here because $q$ is smooth, see \cite{Trogdon2014}), $k$ might be quite small. Note that this method is yet to be implemented in the current code.}  See Section~\ref{sec:Gauss} for a discussion of some nuances for the truncation of the system to a finite one. Then \eqref{eq:large-x-ft} implies
\begin{align}\label{eq:qc0}
    \hat q(k) \approx c_0 \hat{\mathfrak r} (k).
\end{align}
And, as we will see, an advantage of this method over the method for the Fourier transform discussed in the next section is that one is free to replace $\mathfrak r$ with another function $\mathfrak f$, and this function can mirror some properties of the function one is trying to transform.  Most importantly, if $\mathfrak f$ and $q$ have some degree of exponential decay, analytic continuations of $\hat q$ into the complex $k$-plane can be computed.

\begin{remark}
Interestingly, in the relation \eqref{eq:qc0}, one first computes
\begin{align*}
    c_0 = \frac{\hat q(k)}{\hat{\mathfrak r} (k)}.
\end{align*}
So, it is (1) important that $\mathfrak r$ has a Fourier transform that vanishes nowhere, and (2) possible to compute the deconvolution by solving an ODE!
\end{remark}

\begin{remark}
\tcr{
The parameter $\nu$ present in the basis $R_{j,\alpha}$ is important.  In \cite{Weideman1994} it was used to reduce the number of terms required to compute the complex error function to a prescribed tolerance.  In a similar manner, one could choose $\nu$ that the decay rate of the right-hand side of \eqref{eq:banded} is optimized, i.e., to make $|q_{|j|}|$ decay faster as $j$ increases.}
\end{remark}

\subsection{The Cauchy integral method for the Fourier transform}\label{sec:cauchy}

Now, considering the preceding method, it is entirely reasonable to suppose that we can evaluate $\hat q(k)$ pointwise (with a high degree of accuracy), and expand:
\begin{align*}
    \hat q(k) = \sum_{j\neq 0} c_j R_{j,0}(k),
\end{align*}
supposing it has such an expansion. General conditions for such an expansion to exist can be found in \cite{Trogdon2014}. It then follows that 
\begin{align*}
    m(x;k) = \frac{1}{2 \pi \I} \sum_{j\neq 0} c_j \int_{-\infty}^\infty R_{j,x}(k') \frac{\D k'}{k' - k},
\end{align*}
and therefore
\begin{align*}
    q(x) = - \frac{1}{\pi} \sum_{j\neq 0} c_j \lim_{|k| \to \infty} \int_{-\infty}^\infty R_{j,x}(k') \frac{k \D k'}{k' - k}.
\end{align*}
Then this limit is computed using \eqref{eq:large-k-rj}, giving
\begin{align}\label{eq:qx}
    q(x) = \begin{cases}
     \displaystyle - \nu \sum_{j\neq 0} c_j |j|& x = 0,\\
     \displaystyle -2 \nu \sum_{\sign(x)j > 0} c_j L_{|j|-1}^{(1)}(2 |x| \nu)  \E^{-|x| \nu}& \text{otherwise},
    \end{cases}
\end{align}
where the empty sum is taken to be zero. \tcr{Here $L_j^{(\alpha)}$ is the generalized Laguerre polynomials of degree $j$ with parameter $\alpha$, see \cite{DLMF}.  This formula should be compared with that in \cite{Weber1980}.} This series can be easily and stably evaluated using Clenshaw's algorithm \cite{Clenshaw1955} in conjunction with the three-term recurrence for Laguerre polynomials.  \tcr{If one could develop a fast Laguerre transform, i.e., a mapping, and its inverse, from function values at the $n$ roots of $L_n^{(\alpha)}$ to coefficients in a interpolatory Laguerre series in $O(n\,\mathrm{polylog}(n))$ operations \`a la \cite{Hale2016},} then the formula \eqref{eq:qx} would give a competitive alternative to the FFT for the whole line Fourier transform (see also \cite{Weber1980}).

The simpler way to present the material in this section would be to simply present the formula for the Fourier transform of $R_{j,0}$.  However, discussing how the Fourier transform and this formula arises out of the limit of a Cauchy integral is important in the inverse scattering step for the AKNS system.

\subsection{Examples}

\subsubsection{Fourier transform of a Gaussian}\label{sec:Gauss}

As a baseline we consider errors encountered in the computation of the Fourier transform of $q(x) = \E^{-x^2}$, $\hat q(k) = \sqrt{\pi} \E^{-k^2/4}$.  In Figure~\ref{f:least-squares-ft-error} we show the errors for solving \tcr{a $n \times 2000$ least-squares system via the pseudo-inverse as $n$ varies from $n = 100$, doubling to $n = 1600$.  There is clearly slow convergence at near $k = 0$ and this shows that there can be a price to pay for adding in more unknowns}.  This can be fixed by taking a square truncation of \eqref{eq:banded} but the rank of this truncation turns out to be affected by $k$ being positive or negative: Suppose $n$ is even.  If $k\leq 0$ then an $n \times n$ truncation should be used, otherwise an $(n+1) \times (n+1)$ truncation should be used.  This produces dramatically better errors as is demonstrated in Figure~\ref{f:ft-error}.  The approach described in Section~\ref{sec:cauchy} is used to approximate $\hat q$ in Figure~\ref{f:lag-ft-error}.  \tcr{Note that adding more rows to the system, without adding more columns, is unnecessary as the matrix is banded below.}

\begin{figure}[tbp]
    \centering
    \begin{subfigure}[b]{.49\textwidth}
         \centering
         \includegraphics[width=\linewidth]{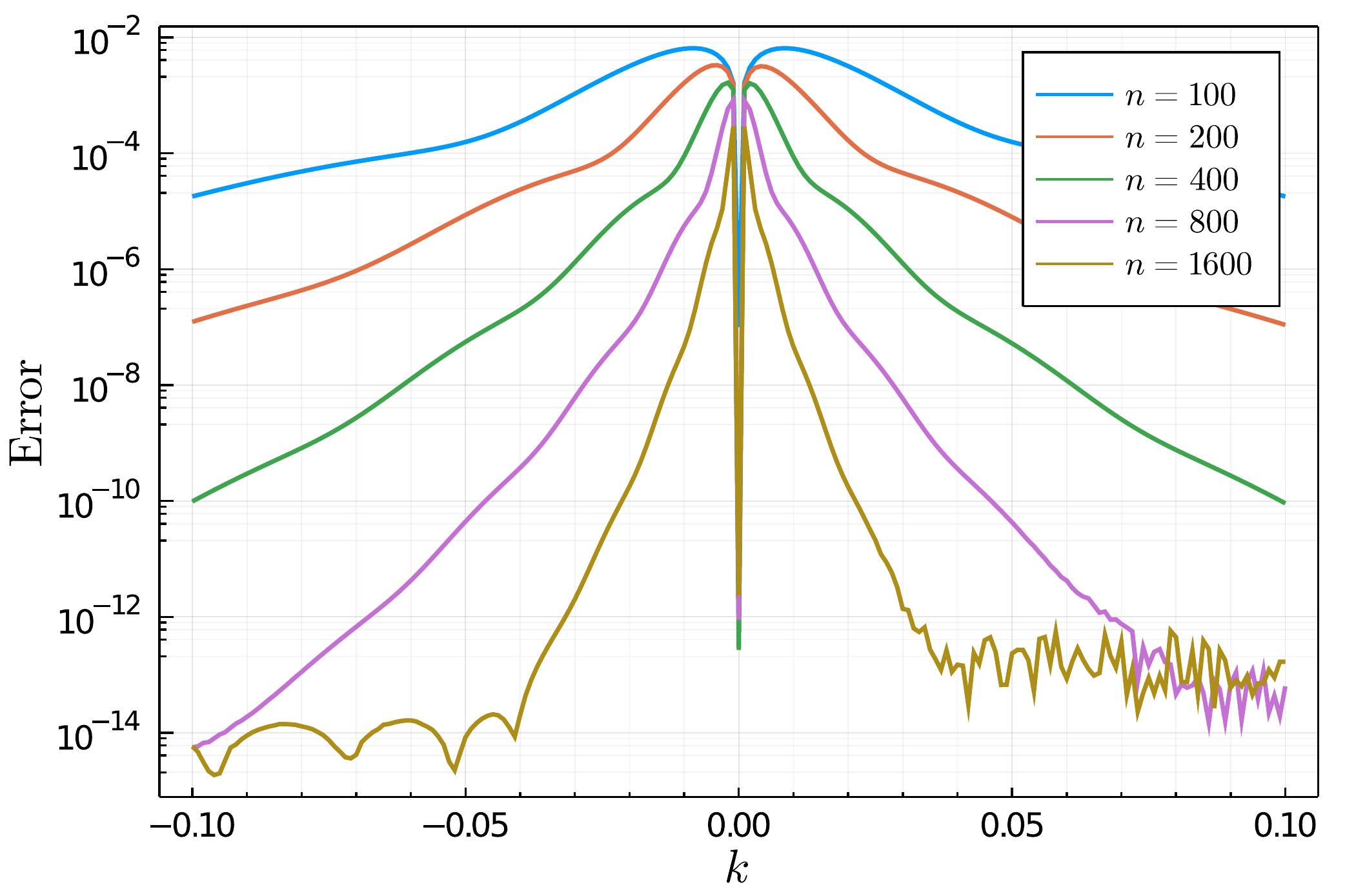}
         \caption{}
         \label{f:least-squares-ft-error}
     \end{subfigure}
     \begin{subfigure}[b]{.49\textwidth}
         \centering
         \includegraphics[width=\linewidth]{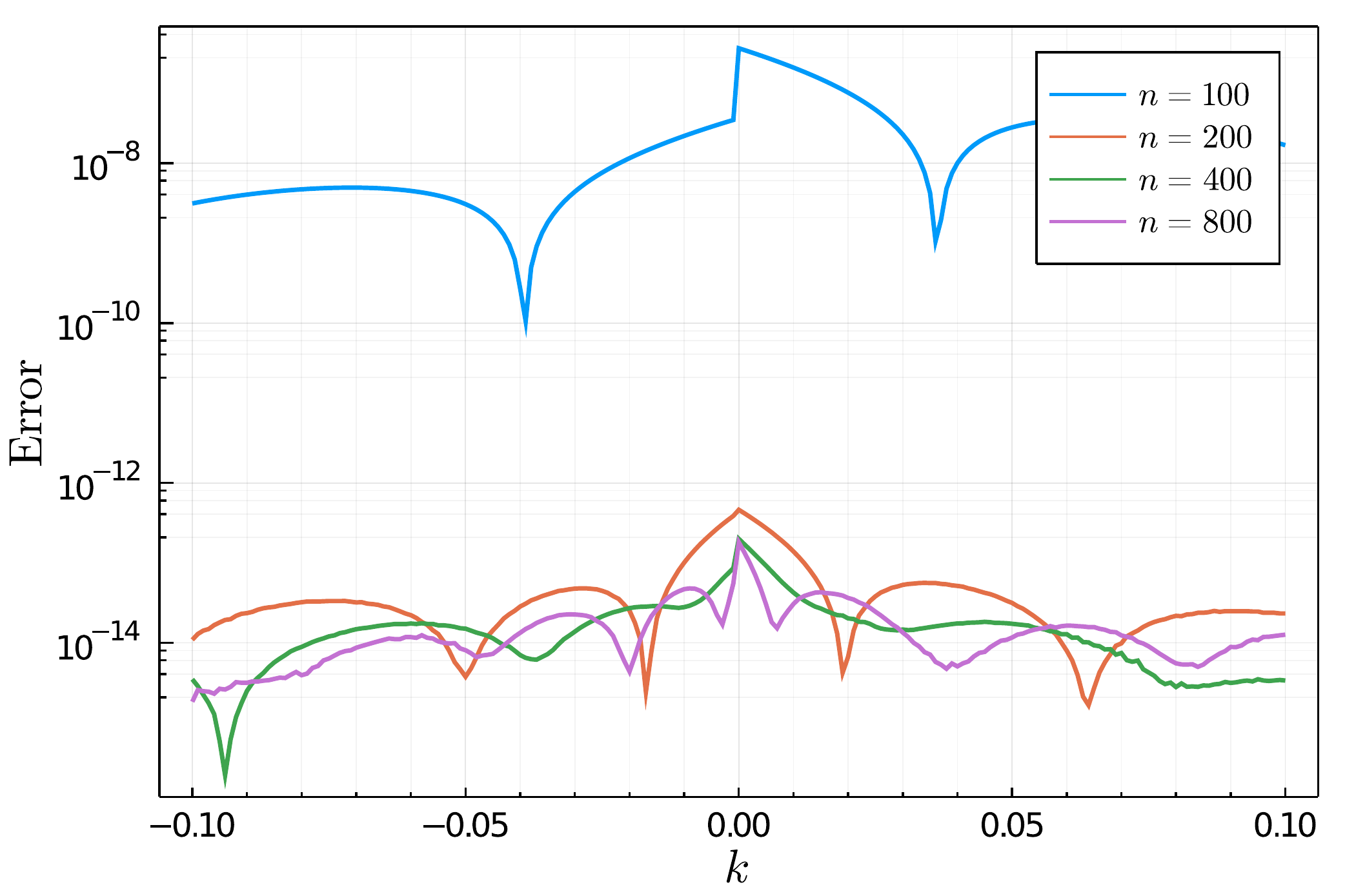}
         \caption{}
         \label{f:ft-error}
       \end{subfigure}
       \caption{(a) The error in computing the Fourier transform of a Gaussian using an $n \times 2000$ finite section of \eqref{eq:banded}.  As $n$ increases the errors decrease but accuracy remains poor near $k = 0$. (b) The error in computing the Fourier transform of a Gaussian using an $n \times n$ finite section of \eqref{eq:banded} if $k \leq 0$ and a $(n+1) \times (n+1)$ finite section for $k > 0$.  As $n$ increases the errors decrease until $n = 400$.}
\end{figure}
\begin{figure}[tbp]
    \centering
    \includegraphics[width=0.49\linewidth]{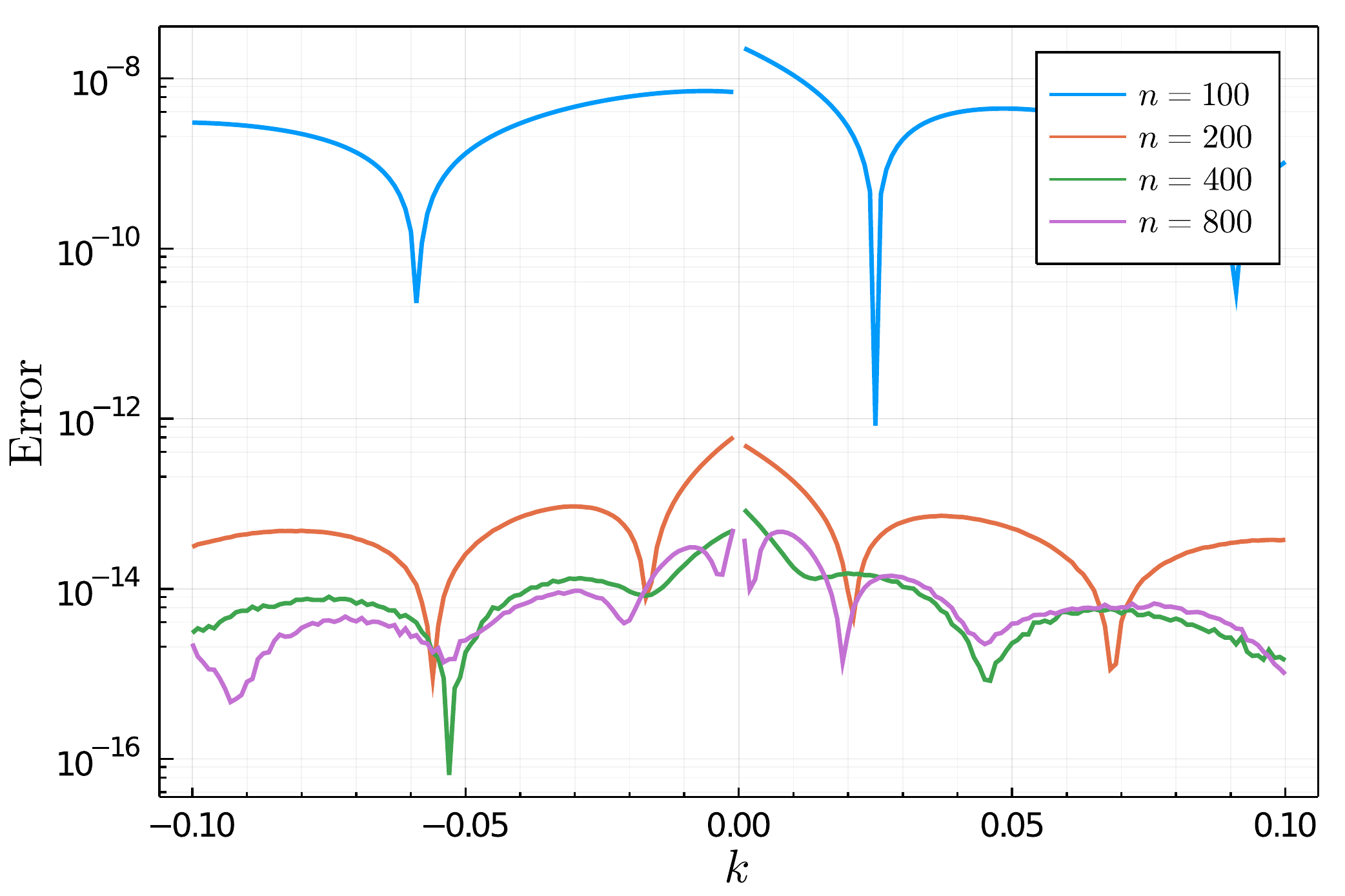}
    \caption{The error in computing the Fourier transform of a Gaussian using \eqref{eq:qx} with $n$ terms in the series.  As $n$ increases the errors decrease until $n = 400$.}
    \label{f:lag-ft-error}
\end{figure}
\begin{figure}[tbp]
    \centering
    \begin{subfigure}[b]{.49\textwidth}
         \centering
         \includegraphics[width=\linewidth]{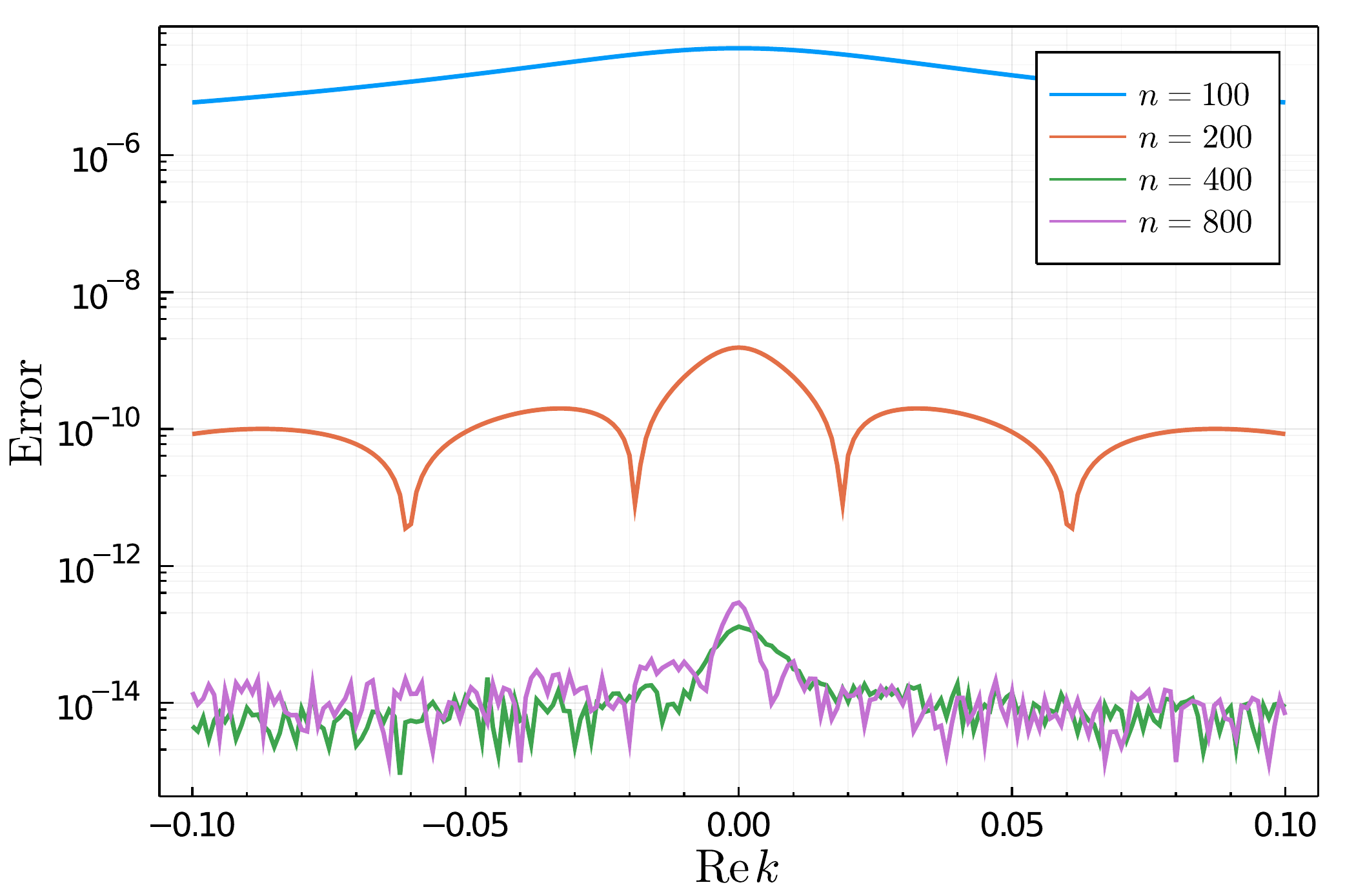}
         \caption{}
         \label{f:ft-gauss-error}
     \end{subfigure}
     \begin{subfigure}[b]{.49\textwidth}
         \centering
         \includegraphics[width=\linewidth]{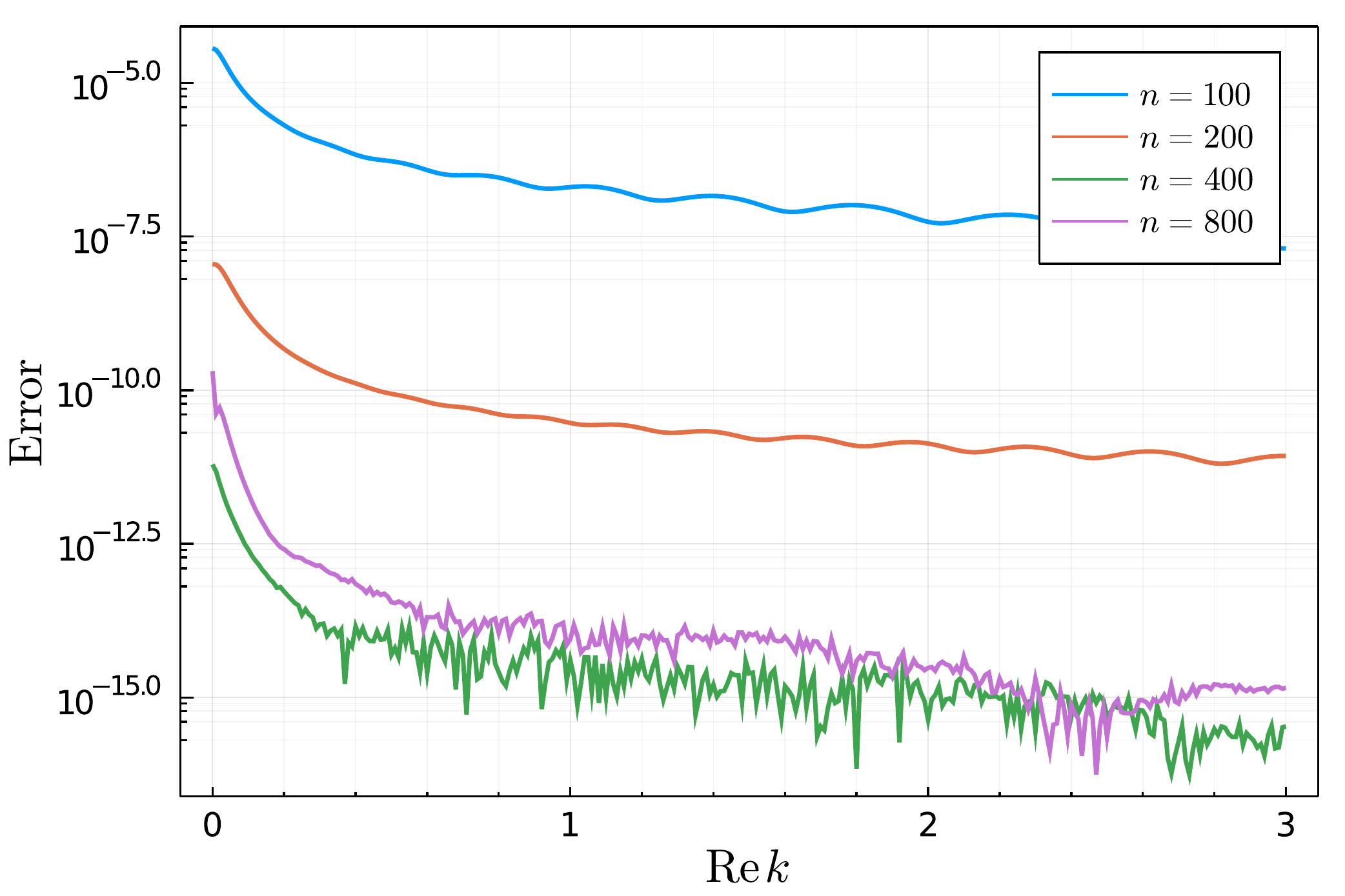}
         \caption{}
         \label{f:ft-gauss-complex-error}
     \end{subfigure}
     \caption{The errors in computing the Fourier transform of a Gaussian in the complex $k$-plane.  (a) The errors along the line $k = \real k + 0.1 \I$.  (b) The errors along the line $k = \real k + 0.1 \I$ for larger values of $k$.  In both (a) and (b) the errors decrease until $n = 400, 800$ and absolute errors are smaller for larger values of $k$.}  \label{f:ft-gauss-complex-ab}
\end{figure}

Since the approach given in Section~\ref{sec:cauchy} is much simpler than that in Section~\ref{sec:ode} it might seem as though the latter has less value despite giving slightly smaller errors.  But it does highlight some phenomena that we will need to be on the look out for in the following sections.  Also, if we wish to evaluate the Fourier transform of a function with some exponential decay in the complex plane, we can \tcr{replace $\mathfrak r$ with a different function, say $\mathfrak g(x) = \E^{-x^2}$}.  This has the added convenience that
\begin{align*}
    \phi(x;k,\mathfrak g) = \frac{\sqrt{\pi}}{2} \E^{-\frac{k^2}{4}} \left( 1 + \mathrm{erf}\left( x + \I \frac{k}{2} \right) \right),
\end{align*}
where $\mathrm{erf}$ is the error function \cite{DLMF}.  To compute this stably for all $x,k \in \mathbb R$, one should rewrite this in terms of the exponentially scaled complementary error function
\begin{align*}
    \mathrm{erfcx}(x) = \E^{x^2} \mathrm{erfc}(x), \quad \mathrm{erfc}(x) + \mathrm{erf}(x) = 1.
\end{align*}
And, interestingly, one of the most effective ways of computing this function is to use the rational basis $\{R_{j,0}\}$, see \cite{Weideman1994} (see also \cite[Section 5.3]{TrogdonSOBook}).  Let $\{g_j\}$ be the coefficients for the expansion of $\mathfrak g$ in the basis $R_{j,0}$:
\begin{align*}
    \mathfrak g(x) = \sum_{j \neq 0} g_j R_{j,0}(x), \quad \vec g = \begin{bmatrix} g_0 \\ g_1 \\ g_{-1} \\ g_2 \\ \vdots \end{bmatrix}.
\end{align*}
Then the first column of \eqref{eq:banded} should be replaced with $\vec g$.  Errors using this methodology to compute the Fourier transform of $f(x) = \E^{-x^2/2}$ off the real axis are shown in Figure~\ref{f:ft-gauss-complex-ab}.  Note that if $f(x) = \E^{-x^2}$ then the numerical method would return the exact (known) transform, and this is the reason for the choice $f(x) = \E^{-x^2/2}$.

\subsubsection{Fourier transform of a rational function}\label{sec:Rat}

Consider something possibly a bit more challenging. We use our methodologies to compute the Fourier transform of the following rational function
\begin{align}\label{eq:rat-fun}
    q(x) = \frac{1}{x - 1 -\I} - \frac{1}{3x + \I}.
\end{align}
Of course,
\begin{align*}
    \hat q(k) = 2 \pi \I \begin{cases}  \E^{-\I k + k}&  k < 0,\\
    \frac 2 3 & k = 0,\\
     \frac{1}{3}\E^{-k/3}  & k > 0. \end{cases}
\end{align*}
The errors resulting from two above approaches in Sections~\ref{sec:ode} and \ref{sec:cauchy} are displayed in Figure~\ref{f:rat-error}.
\begin{figure}[tbp]
    \centering
    \begin{subfigure}[b]{.49\textwidth}
         \centering
         \includegraphics[width=\linewidth]{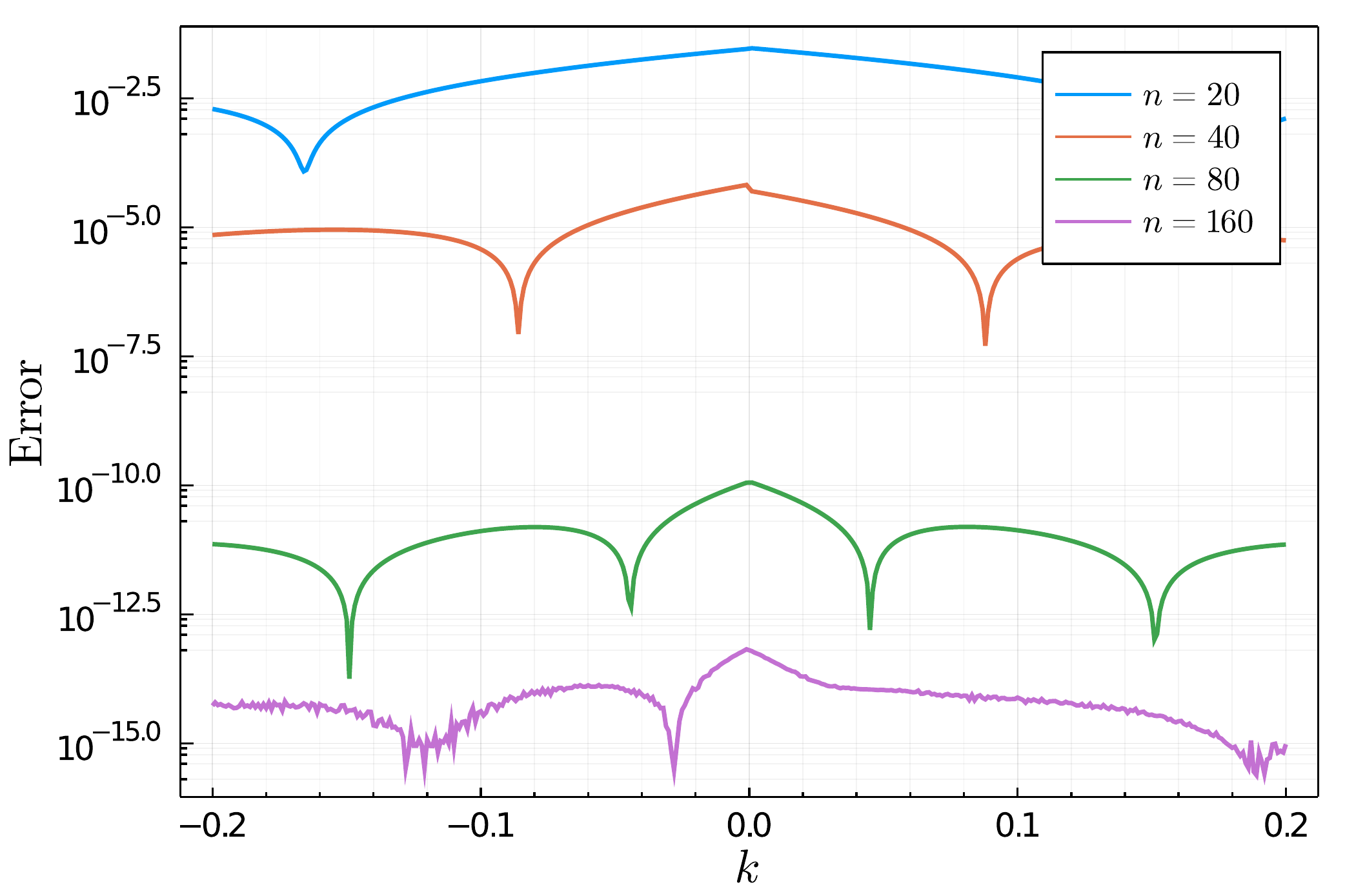}
         \caption{}
         \label{f:ft-rat-error}
     \end{subfigure}
     \begin{subfigure}[b]{.49\textwidth}
         \centering
         \includegraphics[width=\linewidth]{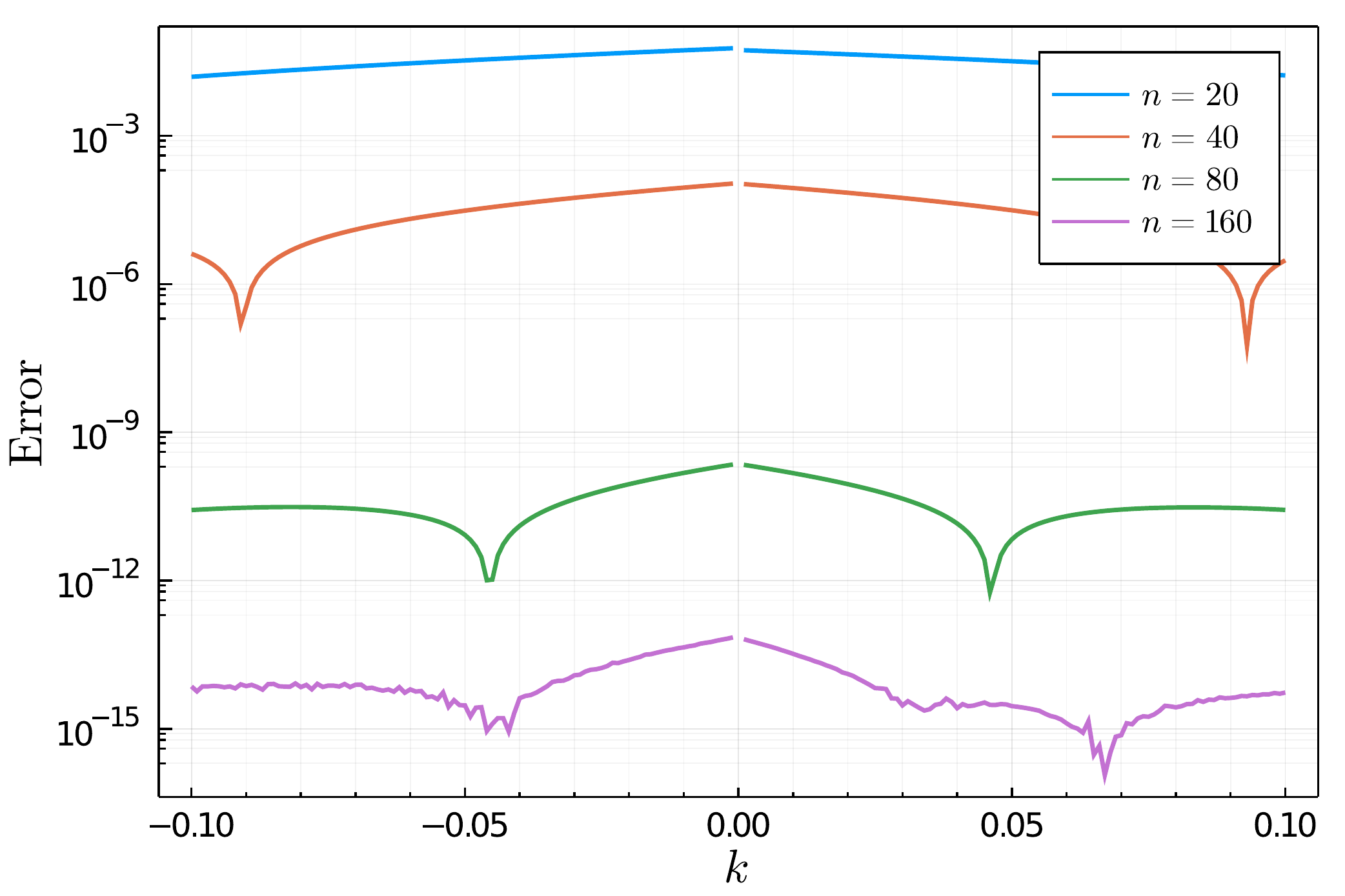}
         \caption{}
         \label{f:lag-rat-ft-error}
       \end{subfigure}
       \caption{(a) The error in computing the Fourier transform of \eqref{eq:rat-fun} using an $n \times n$ or $(n+1) \times (n+1)$ finite section of \eqref{eq:banded}.  As $n$ increases the errors decrease. (b) The error in computing the Fourier transform of \eqref{eq:rat-fun} using \eqref{eq:qx}. As $n$ increases the errors decrease. } \label{f:rat-error}
\end{figure}

\section{Scattering and inverse scattering}\label{sec:scattering}

Before we perform any numerical computations on the AKNS system we introduce the scattering and inverse scattering problems.

The process of \emph{scattering} amounts to computing the requisite quantities, as functions of the $k$ variable, such that the potentials $q$ and $r$ are uniquely, and conveniently, recovered from these quantities.  We will define two reflection coefficients $\rho_1,\rho_2$ and two collections $\{(z_j^\pm, c_j^\pm)\}_{j=1}^{n_\pm}$, \tcr{$n_\pm \in \mathbb N \cup\{0\}$} one for each choice of $\pm$ where
\begin{align*}
    \rho_j &: \mathbb R \to \mathbb C, \quad j = 1,2,\\
    z_j^\pm &\in \mathbb C^\pm, \quad j = 1,2,\ldots,n_\pm.
\end{align*}

Consider two solutions $\mu^\pm(x;k) = \mu^\pm(x;k,q,r) $ of \eqref{eq:AKNS}, normalized at $\pm \infty$ so that
\begin{align*}
    \lim_{x \to \pm \infty} \mu^\pm(x;k) \begin{bmatrix} \E^{\I k x} & 0 \\ 0 & \E^{- \I k x} \end{bmatrix} = I.
\end{align*}
For such solutions to exist we need to assume that $r,q$ are integrable \cite[Lemma 3.1]{TrogdonSOBook}.  We then examine the solutions column-by-column.  For $x < 0$, since the first column $\mu^-_1$ of $\mu^-$ satisfies $\mu_1^-(x;k) \sim \E^{-\I k x} \begin{bmatrix} 0 \\ 1 \end{bmatrix}$ and this exponent decays in the upper-half plane, it is reasonable to expect that for fixed $x$, $\mu^-_1(x;k)$ has an analytic extension to ${\mathbb C^+}$.  If we additionally impose that $q,r \in L^2(\mathbb R)$ it follows that for each fixed $x$
\begin{align*}
    \E^{\I k x} \mu_1^-(k;x)  - \begin{bmatrix} 0 \\ 1 \end{bmatrix},
\end{align*}
is in the range of the Cauchy integral \cite[Lemma 3.5]{TrogdonSOBook}, i.e.,
\begin{align*}
    \frac{1}{2\pi \I} \int_{-\infty}^\infty \frac{u_1^-(k';x)}{k' -k} \D k = \E^{\I k x} \mu_1^-(k;x)  - \begin{bmatrix} 0 \\ 1 \end{bmatrix},
\end{align*}
for a vector-valued function $u_1^-(\cdot;x) \in L^2(\mathbb R)$.   Similar considerations apply to the second column $\mu^-_2$ of $\mu^-$ but now for $\mathbb C^-$.  The implications for $\mu^+$ are reversed, with the first column extending into the lower-half plane and the second column extending into the upper-half plane.

\begin{definition}
The scattering matrix $S(k) = S(k;q,r)$ is given by
\begin{align*}
    S(k) = \lim_{x \to \infty} \begin{bmatrix} \E^{\I k x} & 0 \\ 0 & \E^{-\I k x} \end{bmatrix} \mu^-(x;k), \quad k \in \mathbb R.
\end{align*}
\end{definition}
It follows from this definition that
\begin{align}\label{eq:S}
    \mu^-(x;k) = \mu^+(x;k)S(k), \quad x,k \in \mathbb R.
\end{align}
We write
\begin{align*}
    S(k) = \begin{bmatrix} a(k) & B(k) \\ b(k) & A(k) \end{bmatrix}.
\end{align*}
Since the coefficient matrix in \eqref{eq:AKNS} is traceless, Liouville's formula implies $\det \mu^\pm = 1$ and therefore $\det S(k) = 1$, giving the relation
\begin{align*}
    a(k) A(k) - b(k) B(k) = 1.
\end{align*}
If we impose some relation between $r$ and $q$ we obtain additional conditions \tcr{and it explains why only $a(k)$ and $b(k)$ are specified in \eqref{eq:sechpot} below.}
\begin{proposition}
Suppose $r(x) = \lambda \bar q(x)$, $\lambda = \pm 1$.  Then
\begin{align*}
    \overline{A(\bar k)} = a(k), \quad \overline{B(\bar k)} = - \lambda b(k).
\end{align*}
\end{proposition}
\begin{proof}
Let $\sigma = \sqrt{-\lambda}$.  Then consider the function
\begin{align*}
     \tcr{\check{\mu}}(x;k) = \begin{bmatrix} 0 & \sigma^{-1} \\ \sigma & 0 \end{bmatrix} \overline{\mu(x;\bar k)} \begin{bmatrix} 0 & \sigma^{-1} \\ \sigma & 0 \end{bmatrix}.
\end{align*}
It then follows by uniqueness of solutions of \eqref{eq:AKNS} that $ \tcr{\check{\mu}} = \mu$.  Sending $x \to \infty$,
\begin{align*}
    S(k)= S(k;q,r)= \begin{bmatrix} 0 & \sigma^{-1} \\ \sigma & 0 \end{bmatrix} \overline{S(\bar k)} \begin{bmatrix} 0 & \sigma^{-1} \\ \sigma & 0 \end{bmatrix} = \begin{bmatrix} 
    \overline{A(\bar k)} & \overline{b(\bar k)} \sigma^{-2}\\
    \overline{B(\bar k)} \sigma^{2} & \overline{a(\bar k)}
    \end{bmatrix}.
\end{align*}
\end{proof}

From \eqref{eq:S} it follows
that 
\begin{align*}
    S(k) = (\mu^+(x;k))^{-1} \mu^-(x;k).
\end{align*}
Recall that the first column of $\mu^-$ can be extended analytically to the upper-half $k$-plane.  The same is true of the first row of $(\mu^+(x;k))^{-1}$ because $\det \mu^+ =1$, implying the same analyticity for $a(k)$.  By a similar argument, $A(k)$ can be analytically extended for $k \in \mathbb C^-$.  

Now, build the sectionally-meromorphic matrix-valued function $M: \mathbb C \setminus \mathbb R \to \mathbb C^{2 \times 2}$
\begin{align*}
    M(k) = M(k;x) = \begin{cases}
    \begin{bmatrix} 
    \frac{\mu_1^-(x;k)}{a(k)} & \mu_2^+(x;k) \end{bmatrix}\begin{bmatrix} \E^{\I k x} & 0 \\ 0 & \E^{-\I k x} \end{bmatrix} & \imag k > 0,\\ \\ 
    \begin{bmatrix} \mu_1^+(x;k) & \frac{\mu_2^-(x;k)}{A(k)} \end{bmatrix} \begin{bmatrix} \E^{\I k x} & 0 \\ 0 & \E^{-\I k x} \end{bmatrix} & \imag k < 0.\\
    \end{cases}
\end{align*}

Denote by $\{z_j^+\}_{j=1}^{n_+}$ (resp., $\{z_j^-\}_{j=1}^{n_-}$)  the zeros of $a(k)$ (resp., $A(k)$) in $\mathbb C^+$ (resp., $\mathbb C^-$). In the cases we consider numerically, these sets are finite and $A(k) \neq 0$, $a(k) \neq 0$ for $k \in \mathbb R$. 

We note that a zero of $a(k)$ at $k = z_j^+$ implies that $\mu_1^-(x;z_j^+) = b_j^+ \mu_2^+(x;z_j^+)$ for some constant $b_j^+$. Therefore
\begin{align*}
    \mathrm{Res}_{k = z_j^+} \frac{\mu_1^-(x;k)}{a(k)} = \frac{\mu_1^-(x;z_j^+)}{a'(z_j^+)} = \frac{b_j^+}{a'(z_j^+)}\mu_2^+(x;z_j^+)
\end{align*}
Similarly, for a constant $b_j^-$ satisfying $ \mu_2^-(x;z_j^-) = b_j^- \mu_1^+(x;z_j^-)$
\begin{align*}
    \mathrm{Res}_{k = z_j^-} \frac{\mu_2^-(x;k)}{A(k)} = \frac{\mu_2^-(x;z_j^-)}{A'(z_j^-)} = \frac{b_j^-}{A'(z_j^-)}\mu_1^+(x;z_j^+).
\end{align*}
This leads us to define
\begin{align*}
    c_j^+ = \frac{b_j^+}{a'(z_j^+)}, \quad c_j^- = \frac{b_j^-}{A'(z_j^-)}.
\end{align*}
Lastly, set $\rho_1(k) = b(k)/a(k)$ and $\rho_2(k) = B(k)/A(k)$.  It can then be established that $M$ solves the following Riemann--Hilbert problem.

\begin{RHproblem}\label{rhp:1}
Find $u(\cdot;x): \mathbb R \to \mathbb C^{2\times 2}$ and $u_j^\pm(x) \in \mathbb C^{2 \times 2}$,  $j = 1,2,\ldots,n_\pm$, such that
$$ M(k;x) = I + \frac{1}{2\pi \I} \int_{-\infty}^\infty \frac{u(k';x)}{k' -k} \D k' + \sum_{j=1}^{n_+} \frac{u_j^+(x)}{k - z_j^+} + \sum_{j=1}^{n_-} \frac{u_j^-(x)}{k - z_j^-},$$
\begin{align*}
    M^+(k;x) &= M^-(k;x) \begin{bmatrix} 1 - \rho_1(k) \rho_2(k) & - \rho_2(k) \E^{-2 \I k x} \\
    \rho_1(k) \E^{2 \I k x} & 1\end{bmatrix},\\
    M^\pm(k;x) &= \lim_{\epsilon \to 0^+} M(k \pm \I \epsilon),
\end{align*}
and
\begin{align*}
    \mathrm{Res}_{k = z_j^+} M(k;x) &= \lim_{k \to z_j^+} M(k;x) \begin{bmatrix} 0 & 0 \\
    c^+_j \E^{2 \I z_j^+ x} & 0 \end{bmatrix},\\
    \mathrm{Res}_{k = z_j^-} M(k;x) &= \lim_{k \to z_j^-} M(k;x) \begin{bmatrix} 0 & c_j^- \E^{-2 \I z_j^- x} \\
    0 & 0 \end{bmatrix}.
\end{align*}
\end{RHproblem}
Since $M$ is formed column-by-column using solutions of \eqref{eq:AKNS}, we can see that each column solves
\begin{align*}
    \frac{\D M}{\D x}(k;x) + \I k [\sigma_3, M(k;x)] &= \begin{bmatrix} 0 & q(x) \\ r(x) & 0 \end{bmatrix} M(k;x), ~~~\sigma_3 = \begin{bmatrix} 1 & 0 \\ 0 & -1 \end{bmatrix}.
\end{align*}
Taking the limit $k\to \infty$ in this formula, gives
\begin{align*}
    \lim_{|k| \to \infty} \I k [\sigma_3, M(k;x)] = \begin{bmatrix} 0 & q(x) \\ r(x) & 0 \end{bmatrix}.
\end{align*}

This Riemann--Hilbert problem is posed using the left scattering data --- $S(k)$ is defined by sending in data from $x = -\infty$ and recording how it proceeds through the potentials $r,q$.   This defines a left scattering map $\mathcal S$
\begin{align*}
    \mathcal S_{\mathrm l}(q,r) = \left(\rho_1,\rho_2, (z_j^+,c_j^+)_{j=1}^{n_+}, (z_j^-,c_j^-)_{j=1}^{n_-} \right).
\end{align*}
One can repeat this process to define the right scattering data by enforcing boundary conditions at $x = + \infty$.  In this setting the corresponding relation to \eqref{eq:S} will involve $S(k)^{-1}$.  But, for simplicity, we define the right scattering map by
\begin{align*}
    \mathcal S_{\mathrm r}(q,r) = \mathcal S_{\mathrm l}(r(-\cdot),q(-\cdot)),
\end{align*}
that is, $\mathcal S_{\mathrm r}$ is simply the left scattering map applied with $q(x),r(x)$ replaced with $r(-x),q(-x)$.

To see the implications of this definition, consider
\begin{align*}
     \tcr{\check{\mu}}(x;k) = \begin{bmatrix} 0 & -1 \\ 1 & 0 \end{bmatrix} \mu^+(-x;k,q,r) \begin{bmatrix} 0 & 1 \\ -1 & 0 \end{bmatrix}.
\end{align*}
It follows that
\begin{align*}
    \frac{\D  \tcr{\check{\mu}}}{\D x} (x;k) = \begin{bmatrix} -\I k & r(-x) \\ q(-x) & \I k \end{bmatrix}  \tcr{\check{\mu}}(x;k).
\end{align*}
And as $x\to -\infty$
\begin{align*}
     \tcr{\check{\mu}}(x;k) \sim \begin{bmatrix} 0 & -1 \\ 1 & 0 \end{bmatrix} \begin{bmatrix} \E^{\I k x} & 0 \\ 0 &  \E^{-\I k x} \end{bmatrix} \begin{bmatrix} 0 & 1 \\ -1 & 0 \end{bmatrix} =  \begin{bmatrix} \E^{-\I k x} & 0 \\ 0 &  \E^{\I k x} \end{bmatrix}.
\end{align*}
It follows that $ \tcr{\check{\mu}}(x;k) = \mu^-(x;k,r(-\cdot),q(-\cdot))$.  And then we can compute the behavior of $ \tcr{\check{\mu}}(x;k)$ as $x \to \infty$ using
\begin{align*}
    \begin{bmatrix} \E^{\I k x} & 0 \\ 0 &  \E^{-\I k x} \end{bmatrix}  \tcr{\check{\mu}}(x;k) = \begin{bmatrix} 0 & -1 \\ 1 & 0 \end{bmatrix} \begin{bmatrix} \E^{-\I k x} & 0 \\ 0 &  \E^{\I k x} \end{bmatrix} \mu^+(-x;k,q,r) \begin{bmatrix} 0 & 1 \\ -1 & 0 \end{bmatrix}.
\end{align*}
This implies that if $S(k,q,r) = \begin{bmatrix} a(k) & B(k) \\ b(k) & A(k) \end{bmatrix}$ then
\begin{align*}
    S&(k;r(-\cdot),q(-\cdot))\\
    &=\lim_{x \to + \infty}\begin{bmatrix} \E^{\I k x} & 0 \\ 0 &  \E^{-\I k x} \end{bmatrix}  \tcr{\check{\mu}}(x;k) = \begin{bmatrix} 0 & -1 \\ 1 & 0 \end{bmatrix} S(k;q,r)^{-1} \begin{bmatrix} 0 & 1 \\ -1 & 0 \end{bmatrix} \\
    &= \begin{bmatrix} a(k) & b(k) \\ B(k) & A(k) \end{bmatrix}.
\end{align*}
In this same notation, it follows that if
\begin{align*}
     \mathcal S_{\mathrm r}(q,r) = \left(\gamma_1,\gamma_2, (w_j^+,d_j^+)_{j=1}^{n_+}, (w_j^-,d_j^-)_{j=1}^{n_-} \right).
\end{align*}
and $S(k)$ is as in \eqref{eq:S} and $b_j^\pm$ are as above then
\begin{align*}
    \gamma_1(k) &=  \frac{B(k)}{a(k)}, \quad \gamma_2(k) = \frac{b(k)}{A(k)}, \quad w_j^\pm = z_j^\pm,\\
    d_j^+ &= -\frac{1}{b_j^+ a'(z_j^+)}, \quad d_j^- = -\frac{1}{b_j^-A'(z_j^-)}.
\end{align*}
This demonstrates that if one can compute, $S(k)$, $z_j^\pm$, $a'(z_j^+), A'(z_j^-)$ and $b_j^\pm$ for every $j$, then computing either scattering map is trivial.  The associated right-scattering data Riemann--Hilbert problem is
\begin{RHproblem}\label{rhp:2}
Find $v(\cdot;x): \mathbb R \to \mathbb C^{2\times 2}$, and $v_j^\pm(x) \in \mathbb C^{2\times 2}$, $j =1 ,2,\ldots,n_\pm$, such that
$$ N(k;x) = I + \frac{1}{2\pi \I} \int_{-\infty}^\infty \frac{v(k';x)}{k' -k} \D k' + \sum_{j=1}^{n_+} \frac{v_j^+(x)}{k - z_j^+} + \sum_{j=1}^{n_-} \frac{v_j^-(x)}{k - z_j^-},$$
\begin{align*}
    N^+(k;x) &= N^-(k;x) \begin{bmatrix} 1 - \gamma_1(k) \gamma_2(k) & - \gamma_2(k) \E^{2 \I k x} \\
    \gamma_1(k) \E^{-2 \I k x} & 1\end{bmatrix},\\
    N^\pm(k;x) &= \lim_{\epsilon \to 0^+} N(k \pm \I \epsilon;x),
\end{align*}
and
\begin{align*}
    \mathrm{Res}_{k = z_j^+} N(k;x) &= \lim_{k \to z_j^+} N(k;x) \begin{bmatrix} 0 & 0 \\
    d^+_j \E^{-2 \I z_j^+ x} & 0 \end{bmatrix},\\
    \mathrm{Res}_{k = z_j^-} N(k;x) &= \lim_{k \to z_j^-} N(k;x) \begin{bmatrix} 0 & d_j^- \E^{2 \I z_j^- x} \\
    0 & 0 \end{bmatrix}.
\end{align*}
\end{RHproblem}
The corresponding recovery formula for this RH problem is given by
\begin{align*}
    \lim_{|k| \to \infty} \I k [\sigma_3, N(k;x)] = \begin{bmatrix} 0 & r(x) \\ q(x) & 0 \end{bmatrix}.
\end{align*}
The importance of Riemann--Hilbert Problem~\ref{rhp:2} is that it allows the exchange of $x \to -x$ by selecting $\gamma_1,\gamma_2$, and $d_j^\pm$, $j = 1,2,\ldots,n_\pm$ appropriately and in Section~\ref{sec:inverse} we develop a method that is effective for $x \geq 0$ and this allows the method to immediately apply for $x < 0$ without further modification.

\begin{remark}
On the surface it may seem as the right scattering data cannot be obtained in terms of the left scattering data because all of the entries of $S(k)$ are not known individually.  But it is indeed possible to recover $S(k)$ from the scattering data. \tcr{For example, in case where $a(k)$ and $A(\bar k)$ do not vanish in the closed upper-half plane, the relation
\begin{align*}
    \frac{1}{a(k)A(k)} = 1 - \frac{b(k) B(k)}{a(k)B(k)} = 1 - \rho_1(k) \rho_2(k),
\end{align*}
demonstrates that 
\begin{align*}
    \psi(k) := \begin{cases} a(k)^{-1} & \imag k > 0,\\
    A(k)^{-1} & \imag k < 0, \end{cases}
\end{align*}
solves the Riemann--Hilbert problem
\begin{align*}
    \psi^+(k) \psi^-(k) = 1 - \rho_1(k) \rho_2(k), \quad \psi(\infty) = 1,
\end{align*}
where $\psi$ is analytic in $\mathbb C \setminus \mathbb R$. And the solution of this problem is given by
\begin{align*}
    \psi(k) = \exp\left( \frac{1}{2 \pi \I} \int_{-\infty}^{\infty} \frac{ \log \left(1 - \rho_1(k') \rho_2(k') \right)}{k' - k}  \right), \quad \imag k \neq 0.
\end{align*}
So the reflection coefficients are enough to recover $a(k)$ and $A(k)$ and then $b(k) = \rho_1(k) a(k)$, $B(k) = \rho_2(k)A(k)$.  Zeros of $a, A$ can be incorporated in a straightforward fashion.}
But we will not need this here, nor will we discuss computing  this recovery.
\end{remark}

\section{Numerical scattering}\label{sec:numerical_scattering}

We now discuss extending the method in Section~\ref{sec:ode} to the system \eqref{eq:AKNS} and compute the scattering maps $\mathcal S_{\mathrm l}$ and $\mathcal S_{\mathrm r}$.  We first transform it to a new system to encode boundary conditions more simply:
\begin{align*}
    \varphi(x;k,r,q) = \varphi(x;k) = \mu^-(x;k) \begin{bmatrix} \E^{\I k x} & 0 \\ 0 & \E^{-\I k x} \end{bmatrix} - I.
\end{align*}
The boundary conditions for $\varphi$ are $\varphi(-\infty;k) = \begin{bmatrix} 0 & 0 \\ 0 & 0 \end{bmatrix}$ and $\varphi$ solves
\begin{align*}
    \frac{\D\varphi}{\D x}(x;k) + \I k \left[ \sigma_3, \varphi(x;k)\right] - \begin{bmatrix} 0 & q(x) \\ r(x) & 0 \end{bmatrix} \varphi(x;k) = \begin{bmatrix} 0 & q(x) \\ r(x) & 0 \end{bmatrix}.
\end{align*}
This matrix-valued ODE can be solved column-by-column.  And the taking $k \mapsto -k$, $(q,r) \mapsto (r,q)$, and interchanging the rows of the solution maps one solution to the other.  So, it suffices to derive a general method for the first column:
\begin{align}\label{eq:varphi1}
    \frac{\D\varphi_1}{\D x}(x;k) - \begin{bmatrix} 0 & q(x) \\ r(x) & 2 \I k \end{bmatrix} \varphi_1(x;k) = \begin{bmatrix} 0  \\ r(x) \end{bmatrix}, \quad \varphi_1(-\infty;k) = \begin{bmatrix} 0 \\ 0 \end{bmatrix}.
\end{align}

Since multiplication by $q$ and $r$ are going to damage the banded structure encountered in Section~\ref{sec:ode}, we exchange $\mathfrak r$ for $\mathfrak g(x) = \E^{-x^2}$, as discussed in Section~\ref{sec:Gauss}.  Then applying the differential operator $\frac{\D}{\D x} - \I k$ to
\begin{align*}
    c_0 \phi(x;k,\mathfrak g) + \sum_{j \neq 0} c_j R_{j,0}(x),
\end{align*}
can be expressed in block form as
\begin{align*}
    \left[\begin{array}{c|c} \vec g & \mathcal D_\nu - \I k I \end{array}\right] \left[\begin{array}{c} c_0 \\\cline{1-1} c_1 \\ c_{-1} \\ c_2 \\ \vdots \end{array} \right].
\end{align*}
We use $\vec q, \vec r$ to denote the expansion coefficients of $q,r$, respectively, in the basis $\{R_{j,0}\}$.  The last piece we need is the coefficients in the expansion:
\begin{align}\label{eq:qphi}
    q(x) \phi(x;k,\mathfrak g) &= \sum_{j\neq 0} g_{q,k,j} R_{j,0}(x)
\end{align}
and we use the notation
\begin{align*}
    \vec{q\mathfrak g}(k) = \begin{bmatrix}
    g_{q,k,1} \\ g_{q,k,-1} \\ g_{q,k,2} \\ g_{q,k,-2} \\ \vdots
    \end{bmatrix}.
\end{align*}

\begin{remark}
One might be worried about the expansion \eqref{eq:qphi} and the decay rate of the coefficients $g_{q,k,j}$  as $|j|$.  But it turns out that $\phi(x;k,\mathfrak g)$ is a non-oscillatory function:  For fixed $k$, $\phi(x;k,\mathfrak g)$ oscillates as $x \to \infty$, but the amplitude of these oscillations decay exponentially as $|k| \to \infty$. This indicates that, depending on how the coefficients are computed, \eqref{eq:qphi} should converge uniformly with respect $k$.
\end{remark}

We alert the reader to the notation \eqref{eq:multmat} so that we can then summarize \eqref{eq:varphi1} in block-matrix form
\begin{align}\label{eq:scatteringsys}
    \left[\begin{array}{c|c|c|c} \vec g & \mathcal D_\nu & - \vec{q\mathfrak g}(2k) & - \mathcal M_{\mathcal I}(\vec q)  \\\cline{1-4}
     - \vec{q\mathfrak g}(0) & - \mathcal M_{\mathcal I}(\vec r)  & \vec g & \mathcal D_\nu - 2\I k I \end{array} \right] \left[\begin{array}{cc} u_0 \\\cline{1-1} u_{1} \\ u_{-1} \\ \vdots \\\cline{1-1} v_0 \\\cline{1-1} v_1 \\ v_{-1} \\ \vdots \end{array} \right] = \left[\begin{array}{cc}  0 \\ \cline{1-1} \vec r \end{array} \right], 
\end{align}
where
\begin{align*}
    \varphi_1(x;k) = \begin{bmatrix} u_0 \phi(x;0,\mathfrak g) \\ v_0 \phi(x;2k,\mathfrak g) \end{bmatrix} + \sum_{j\neq 0} R_{j,0}(x) \begin{bmatrix} u_j \\ v_j \end{bmatrix}.
\end{align*}
In practice, it should be beneficial to interlace the unknowns and the right-hand side vector --- to de-block the system --- and apply the adaptive QR algorithm.  But this has yet to be implemented in the current codes.  Instead, each block is truncated to have $m$ rows and the bi-infinite operators are truncated to have $m-1$ columns. The resulting system is solved by least-squares.  This methodology also allows us to compute the second column $\varphi_2$ of $\varphi$ as:
\begin{align*}
    \varphi_2(x;k) = \begin{bmatrix} w_0 \phi(x;-2k,\mathfrak g) \\ y_0 \phi(x;0,\mathfrak g) \end{bmatrix} + \sum_{j\neq 0} R_{j,0}(x) \begin{bmatrix} w_j \\ y_j \end{bmatrix}.
\end{align*}
Then examining the large $x$ behavior of $\varphi_1,\varphi_2$, the scattering matrix $S(k)$ is simply given by
\begin{align*}
    S(k) = \begin{bmatrix} 1 + u_0 \hat {\mathfrak g}(0) & w_0\hat {\mathfrak g}(-2k) \\ v_0\hat {\mathfrak g}(2k) & 1 + y_0\hat {\mathfrak g}(0) \end{bmatrix}.
\end{align*}
Note that here we have ignored errors in the truncation of the system to a finite one and these errors will be analyzed empirically.

\subsection{Computing $z_j^\pm$, $c_j^\pm$ and $d_j^\pm$}

Now that $S(k)$ can be constructed, we need to discuss a method for constructing the discrete scattering data.  One way to do this is to turn \eqref{eq:AKNS} into a eigenvalue equation.  We have
\begin{align}\label{eq:ode-evalprob}
    \left( \I \sigma_3  \frac{\D}{\D x}  - \I \sigma_3 \begin{bmatrix} 0 & q(x) \\ r(x) & 0 \end{bmatrix} \right) \mu(x;k) = k \mu(x;k).
\end{align}
And $a(z_j^+) = 0$ implies that the first column of $\mu^-(x;z_j^+)$ is an $L^2(\mathbb R)$ eigenfunction of the differential operator on the right-hand side for eigenvalue $k = z_j^+$.  Similarly, $A(z_j^-) = 0$  implies that the second column of $\mu^+(x;z_j^+)$ is an $L^2(\mathbb R)$ eigenfunction with eigenvalue $k = z_j^-$.  So, we discretize the operator as
\begin{align}\label{eq:evalprob}
    \left[\begin{array}{c|c} \I  \mathcal D_{\nu} & -\I \mathcal M_{\mathcal I}(q) \\\cline{1-2}
     \I \mathcal M_{\mathcal I}(r) & -\I \mathcal D_{\nu} \end{array} \right].
\end{align}
Replacing each operator with a square $n\times n$ finite-section truncation, we obtain a finite-dimensional eigenvalue problem, and for $n$ sufficiently large the eigenvalues off the real axis are expected to be good approximations of $z_j^\pm$.  To provide a check, and a refinement, we note that \cite[(3.16)]{TrogdonSOBook}
\begin{align*}
    a(k) = 1 + \frac{1}{2 \pi \I} \int_{-\infty}^\infty \frac{a(k') - 1}{k' - k} \D k.
\end{align*}
This implies
\begin{align}\label{eq:arep}
    a(k) = 1 + \sum_{j > 0} a_j R_{j,0}(k),
\end{align}
and both $a$ and its derivative $a'$ are easily computed in the entire upper-half plane.  For eigenvalues in the upper-half plane, approximated via truncations of \eqref{eq:evalprob}, this allows one to use a few steps of Newton's method to refine, if necessary.  Additionally,
\begin{align*}
    A(k) = 1 + \sum_{j < 0} A_j R_{j,0}(k).
\end{align*}
can be used to refine eigenvalues in the lower-half plane.

\begin{remark}
Note that \eqref{eq:arep} implies that $a(k) = 1 + O(1/k)$ as $k \to \infty$ and a rational basis, or at least a basis of functions that decay slowly at infinity, is required to represent $a(k)$ in the complex plane.  
\end{remark}

To complete the computation of the scattering data, we note that since $\phi(x;k,\mathfrak g)$ is an entire function of $k$, \eqref{eq:scatteringsys} can, in principle, be solved for complex $k$.  It also follows that
\begin{align*}
    \phi(-x;-k, - q(- \cdot), -r(-\cdot) ) = \mu^+(x;k,q,r) \begin{bmatrix} \E^{\I k x} & 0 \\ 0 & \E^{-\I k x} \end{bmatrix} - I
\end{align*}
so that when columns of $\mu^\pm$ are proportional, i.e., when $a(k), A(k) = 0$, then the constants of proportionality can be deduced by computing the individual solutions.  This allows one to compute $b_j^\pm$ once $z_j^\pm$ are known. The constants $c_j^\pm$ are then easily computed from $b_j^\pm$, $a'(z_j^+)$, $A'(z_j^-)$.

\subsection{Examples}

\subsubsection{Modulated $\mathrm{sech}$  potential}

The best point of comparison for the computation of the scattering maps for the AKNS system is the case discussed in \cite{Tovbis2004a}.  The authors show that the potential
\begin{align}\label{eq:sechpot}
    q(x) = - \I \tcr{\amp} \mathrm{sech}(x) \exp(-\I \gamma \tcr{\amp} \log \mathrm{cosh}(x)), \quad \tcr{\amp} > 0, \gamma \in \mathbb R,
\end{align}
gives rise to the scattering data
\begin{align*}
    a(k) &= \frac{\Gamma(w(k)) \Gamma(w(k) - w_- - w_+)}{\Gamma(w(k) - w_+) \Gamma(w(k) - w_-)},\\
    b(k) &= \I \tcr{\amp}^{-1} 2^{-\I \gamma \tcr{\amp}} \frac{\Gamma(w(k)) \Gamma(1 - w(k) + w_- + w_+)}{\Gamma(w_+) \Gamma(w_-)},\\
    w(k) &= -\I z - \tcr{\amp} \gamma \frac{\I}{2} + \frac 1 2, ~~ w_+ = -\I \tcr{\amp} \left( T + \frac{\gamma}{2} \right),\\
    w_- &= \I \tcr{\amp} \left( T - \frac{\gamma}{2} \right), ~~ T = \sqrt{\frac{\gamma^2}{4} -1}.
\end{align*}
in the case that $r(x) = -\overline{q(x)}$, which corresponds to the scattering problem associated with the focusing nonlinear Schr\"odinger equation. The zeros of $a(k)$ are given by
\begin{align}\label{eq:zjs}
    z_j^+ = \tcr{\amp} T - \I (j - 1/2), \quad j = 1,2,\ldots, n_+ := \lfloor 1/2 + \tcr{\amp} |T| \rfloor.
\end{align}
Also $b_j^+ = b(z_j^+)$ and $c_j^- = - \overline{c_j^+}$, $d_j^- = - \overline{d_j^+}$.

When $\tcr{\amp} = 1.65$ and $\gamma = 0.1$, we compare the computed functions $a(k), b(k), \rho_1(k)$ with their analytical expression in Figure~\ref{f:sech-error}.
\begin{figure}[tbp]
    \centering
    \begin{subfigure}[b]{.49\textwidth}
         \centering
         \includegraphics[width=\linewidth]{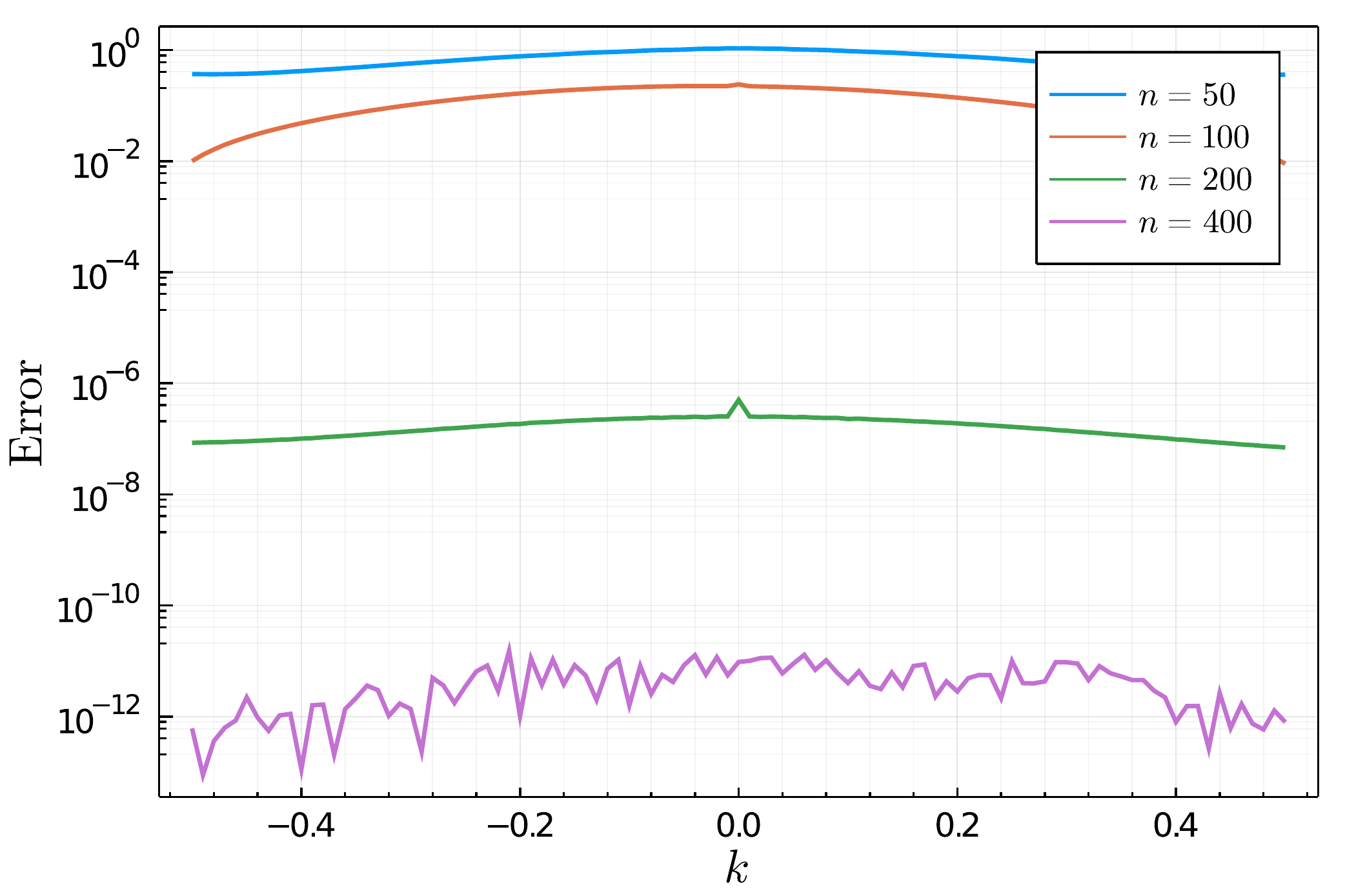}
         \caption{}
         \label{f:sech-a-error}
     \end{subfigure}
     \begin{subfigure}[b]{.49\textwidth}
         \centering
         \includegraphics[width=\linewidth]{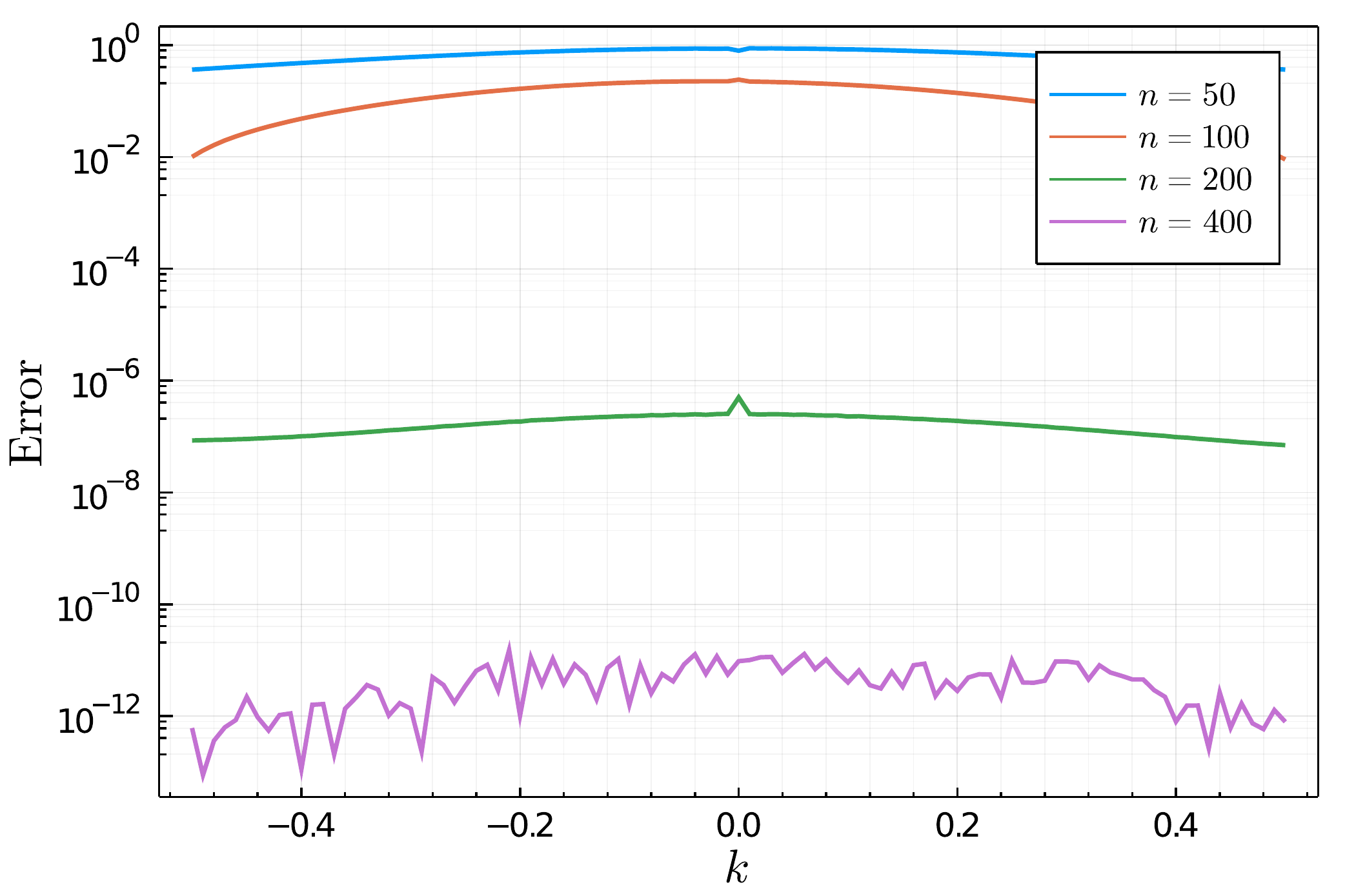}
         \caption{}
         \label{f:sech-b-error}
     \end{subfigure}
     \begin{subfigure}[b]{.49\textwidth}
         \centering
         \includegraphics[width=\linewidth]{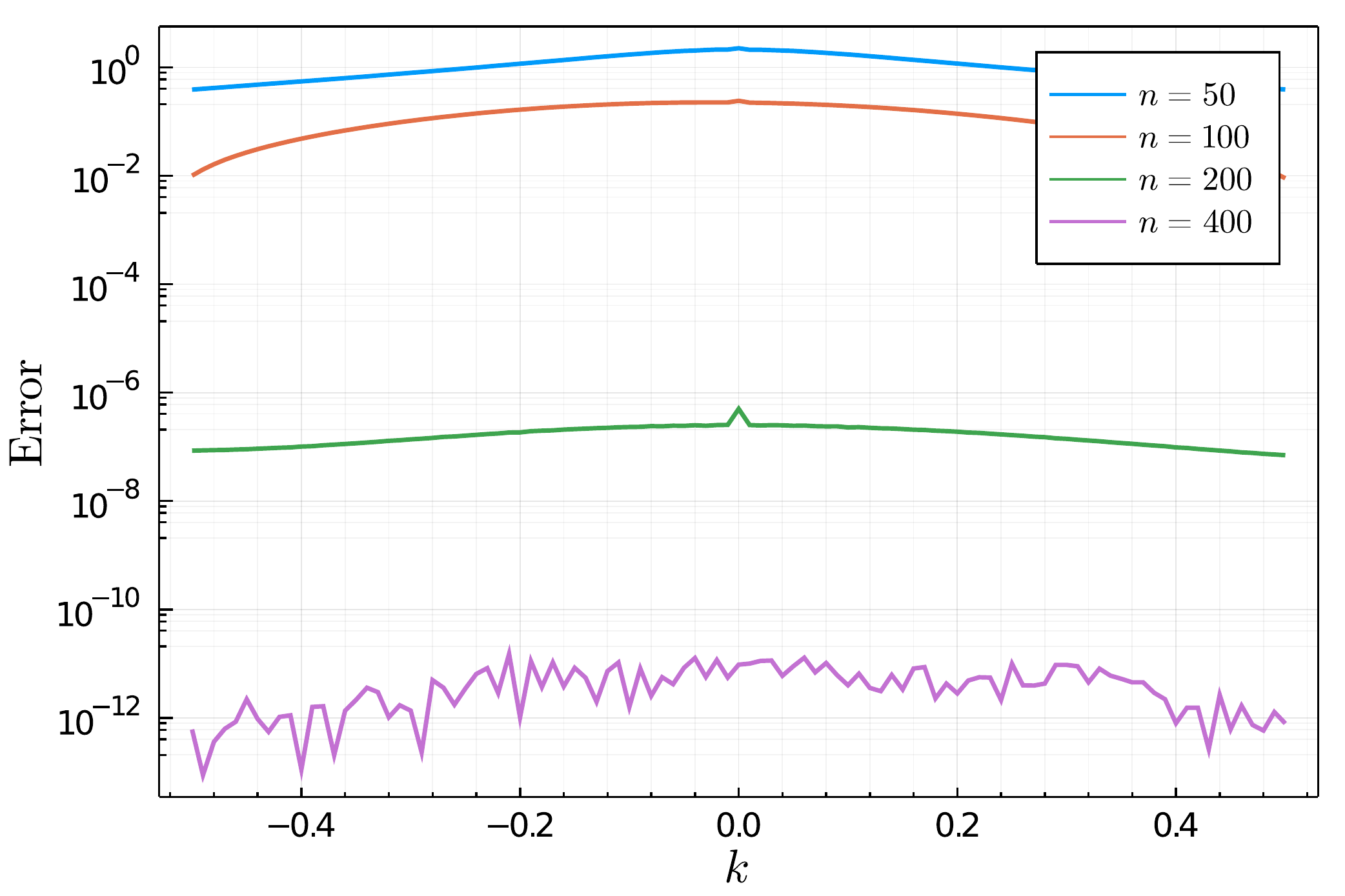}
         \caption{}
         \label{f:sech-reflection-error}
     \end{subfigure}
     \caption{The numerical verification of computing $a(k),b(k)$ and $\rho_1(k)$ for \eqref{eq:sechpot} with $\tcr{\amp} = 1.65$ and $\gamma = 0.1$ for various finite section dimensions $n$.  Errors decrease as $n$ increase in all cases. (a) The error in computing $a(k)$ with an $(2n + 200) \times 2n$ finite section of \eqref{eq:scatteringsys} as $n$ varies.    (b) The error in computing $b(k)$ with an $(2n + 200) \times 2n$ finite section of \eqref{eq:scatteringsys} as $n$ varies. (c) The error in computing $\rho_1(k)$ with an $(2n + 200) \times 2n$ finite section of \eqref{eq:scatteringsys} as $n$ varies. } \label{f:sech-error}
\end{figure}

\subsection{Gaussian data}
Consider the case
\begin{align}\label{eq:qr-gauss}
  \begin{split}
    q(x) &= \E^{-x^2},\\
    r(x) &= -2 \E^{-x^2 +\I x}.
  \end{split}
\end{align}
Here $r(x)$ and $q(x)$ share no obvious relation.  In Figure~\ref{f:Phi} we plot the four components of $\varphi(x;1)$ as functions of $x$.
\begin{figure}[tbp]
    \centering
    \includegraphics[width=\linewidth]{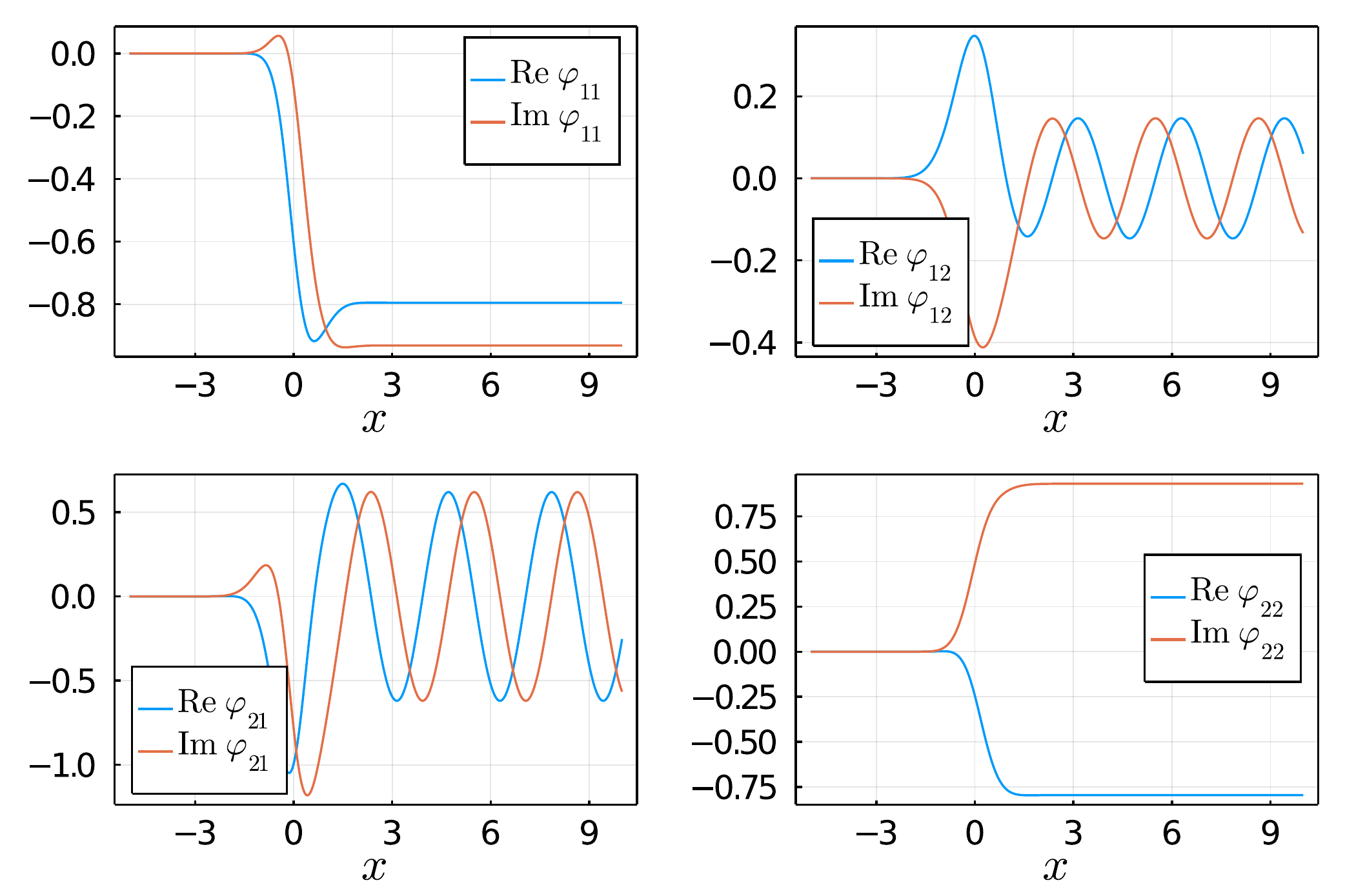}
    \caption{The four components of the numerially computed solution $\varphi(x;1)$ where $q,r$ are given by \eqref{eq:qr-gauss}.}
    \label{f:Phi}
\end{figure}
And in Figure~\ref{f:rho1rho2} we plot the four reflection coefficients $\rho_1,\rho_2,\gamma_1,\gamma_2$.
\begin{figure}[tbp]
    \centering
    \includegraphics[width=\linewidth]{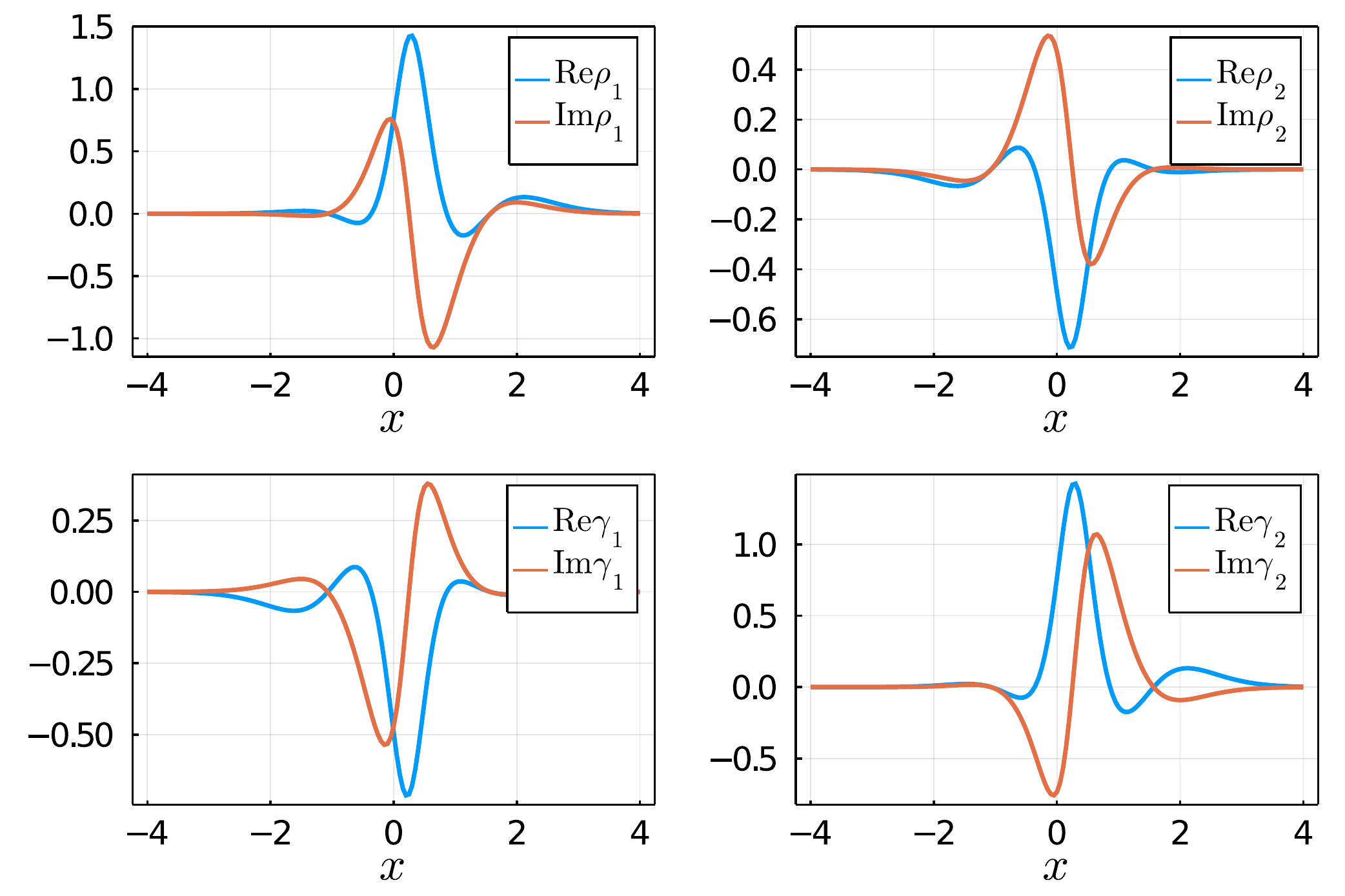}
    \caption{The numerially computed functions $\rho_1,\rho_2,\gamma_1,\gamma_2$ when $q,r$ are given by \eqref{eq:qr-gauss}.}
    \label{f:rho1rho2}
\end{figure}
For this data we find two eigenvalues 
\begin{align*}
    z_1^+ &\approx 0.25 + 0.517003899379\I,\\
    z_1^- &\approx 0.25 - 0.517003899379\I.
\end{align*}
And the associated norming constants are given by
\begin{align*}
    c_1^+ &\approx 0.4100036 - 1.6000283\I , \quad d_1^+ \approx 0.2050018 + 0.8000141\I,\\
    c_1^- &\approx0.2050018 - 0.8000141\I, \quad d_1^- \approx 0.4100036 + 1.6000283\I.
\end{align*}
\begin{remark}
In computing these quantities the parameter $\nu$ that is present in the basis functions appears to be critical to obtain accurate results.  Specifically, $\nu = 12$ produces very accurate approximations of $z_j^\pm$.  \tcr{The parameter $\nu$ is most effectively chosen so that it tracks the decay rate of the true eigenfunctions (the basis functions decay lower as $\nu$ increases).  But it is difficult to know \emph{a priori} how to choose $\nu$.}

Additionally, the choice of $n$ and $m$ for the truncation of \eqref{eq:scatteringsys} to a finite system are quite important.  For a given $n$, numerical experiments seem to indicate that the optimal value of $m$ depends on $k$.  This is something that the adaptive QR algorithm should be able to detect.  As this is not implemented in the current work, accuracies are often limited to $\approx 10^{-8}$ for anomalous values of $k$.  This is a shortcoming of this method that would need to be rectified for it to be competitive with that in \cite{TrogdonSONLS}, for example.  Patching this shortcoming and exploring the implications of varying $\nu$ are left for future work.  It is also a hope that this methodology can be used to compute scattering data with high-precision.  A possible way forward for this is discussed in more detail in Section~\ref{sec:conc}
\end{remark}

\section{Numerical inverse scattering}\label{sec:inverse}

Before we turn to the specifics of solving the inverse problem we review some facts stated in \cite{Trogdon2013} about the GMRES algorithm. GMRES \cite{GMRES-original} can be applied, in abstract form, to the linear operator equation on a Hilbert space $\mathbb H$
\begin{align*}
    \mathcal A u = f, \quad  \mathcal A : \mathbb H \to \mathbb H, \quad f \in \mathbb H,
\end{align*}
provided the following conditions hold:
\begin{enumerate}
\item There exists a collection of function $S = \{u_\alpha\}_{\alpha \in A} \subset \mathbb H$ such that if $v \in S$, then the expansion coefficients $\{c_\beta\}$,  $\mathcal Av = \sum_{\beta \in B} c_\beta u_\beta$, can be computed as a finite sum.
\item $f$ can be expressed as a finite linear combination of elements of $S$.
\item For $v,w \in S$, the inner product $\langle v, w\rangle$ is computable \tcr{exactly}.
\end{enumerate}
These are conditions required for the GMRES iterations to be performed \underline{exactly}.  Most of the time, approximations are built in, in the form of either replacing $f$ with an approximation, replacing the inner product with an approximation or approximating the action of the operator $\mathcal A$.

We first suppose that the scattering data does not include any discrete spectrum $z_j^\pm$.  To solve the inverse problem, i.e., to reconstruct $r,q$ from the scattering data we use the approach in \cite{Trogdon2013} where it was noted that the GMRES algorithm, using the functions $R_{j,\pm 2x}(k)$, applied to the singular integral equation
\begin{align} \label{eq:SIE}
    \mathcal C^+u(k)\begin{bmatrix} 1 & 0 \\ -\rho_1(k) \E^{2 \I k x} & 1 \end{bmatrix} - \mathcal C^-u(k)&\begin{bmatrix} 1 & -\rho_2(k) \E^{-2 \I k x}  \\ 0 & 1 \end{bmatrix} \\
    &= \begin{bmatrix} 0 & -\rho_2(k) \E^{-2 \I k x} \\ \rho_1(k) \E^{2 \I k x} & 0 \end{bmatrix}\notag
\end{align}
converges rapidly, provided that $x \geq 0$. The Cauchy operators $\mathcal C^\pm$ are defined in \eqref{eq:cauchy-def}.

To see how this integral equation is related to Riemann--Hilbert Problem~\ref{rhp:1}, we first note that
\begin{align*}
    \begin{bmatrix} 1 & -\rho_2(k) \E^{-2 \I k x}  \\ 0 & 1 \end{bmatrix}&\begin{bmatrix} 1 & 0 \\ \rho_1(k) \E^{2 \I k x} & 1 \end{bmatrix} \\ &= \begin{bmatrix} 1 - \rho_1(k) \rho_2(k) & - \rho_2(k) \E^{-2 \I k x} \\
    \rho_1(k) \E^{2 \I k x} & 1\end{bmatrix} =: J(k;x).
\end{align*}
And in supposing that we have no poles in the upper- or lower-half planes, the statement of Riemann--Hilbert Problem~\ref{rhp:1} allows us to write
\begin{align*}
    \mathcal C^+ u(k) &- \mathcal C^- u(k) J(k;x) = J(k;x)-I.
\end{align*}
Multiplying this equation on the right by $\begin{bmatrix} 1 & 0 \\ -\rho_1(k) \E^{2 \I k x} & 1 \end{bmatrix}$ gives the desired integral equation.  This factorization was used in \cite{deift-zhou:mkdv} to show that \eqref{eq:SIE} is uniquely solvable provided $\|\rho_1\|_\infty, \|\rho_2\|_\infty < 1$ and therefore Riemann--Hilbert Problem~\ref{rhp:1} has a unique solution (again, provided the discrete spectrum is empty).

The main improvement the current paper makes to the methodology in \cite{Trogdon2013} is that which is discussed in Appendix~\ref{sec:cauchy-app}, and specifically given in \eqref{eq:cauchy_better} and \eqref{eq:cauchy_better2}.  This approach allows for the Cauchy operators $\mathcal C^\pm$ to be applied to
\begin{align*}
    \sum_{|j| \leq m} c_j R_{j,\alpha}(k),
\end{align*}
resulting in
\begin{align*}
    \sum_{|j| \leq m} c^0_j R_{j,0}(k) + \sum_{|j| \leq m} c^\alpha_j R_{j,\alpha}(k),
\end{align*}
where the coefficients $c^0_j,c^\alpha_j$, $j = -m,\ldots,m$ can be computed in an $\alpha$-independent $O(m^2)$ operations, for any $\alpha \in \mathbb R.$

Once an approximation $u$ of the solution of \eqref{eq:SIE} is known, then the recovery formula is easily computed via
\begin{align*}
    \lim_{k \to \infty} k (M(k;z) - I) = - \frac{1}{2 \pi \I} \int_{-\infty}^\infty u(k) \D k,
\end{align*}
and this can be easily evaluated as in the case \eqref{eq:qx} provided that the approximation $u$ is given as finite sum of $R_{j,\alpha}(k)$ for various choices of $j$ and $\alpha$.

The role of Riemann--Hilbert Problem~\ref{rhp:2} is that it can be written in the form
\begin{align} \label{eq:SIE2}
    \mathcal C^+v(k)\begin{bmatrix} 1 & 0 \\ -\gamma_1(k) \E^{-2 \I k x} & 1 \end{bmatrix} - \mathcal C^-v(k)&\begin{bmatrix} 1 & -\gamma_2(k) \E^{2 \I k x}  \\ 0 & 1 \end{bmatrix} \\
    &= \begin{bmatrix} 0 & -\gamma_2(k) \E^{2 \I k x} \\ \gamma_1(k) \E^{-2 \I k x} & 0 \end{bmatrix}\notag,
\end{align}
giving a rapidly converging method for $x \leq 0$.

To get a sense of why these methods converge rapidly, we describe how one can determine that the first row of $u(k)$ is of the form
\begin{align}\label{eq:form}
    \begin{bmatrix} u_1(k) & u_2(k) \end{bmatrix},
\end{align}
where $u_1(k)\E^{-2 \I k x},u_2(k) \E^{2 \I k x}$ are non-oscillatory functions --- they are readily approximated in the basis $\{R_{j,0}\}$. Indeed, since $\mathcal C^+ - \mathcal C^- = I$ we write
\begin{align}\label{eq:first-row}
    \begin{bmatrix} u_1(k) & u_2(k) \end{bmatrix}  &-  \begin{bmatrix} - \rho_1(k) \E^{2 \I k x}  \mathcal C^+ u_2(k) &  \rho_2(k) \E^{-2 \I k x}  \mathcal C^- u_1(k) \end{bmatrix} \\
    &= \begin{bmatrix} 0 & -\rho_2(k) \E^{-2 \I k x} \end{bmatrix}.\notag
\end{align}
If we write $u_2(k) = v(k) \E^{-2 \I k x}$ where $v$ is non-oscillatory then Theorem~\ref{Theorem:CauchyAction} implies that $\mathcal C^+ u_2$ is non-oscillatory.  Similarly, if $u_1(k) = w(k) \E^{2 \I k x}$  where $w$ is non-oscillatory then $\mathcal C^- u_1$ is non-oscillatory.  Thus \eqref{eq:first-row} can be rewritten in terms of non-oscillatory functions:
\begin{align}\label{eq:first-row-mod}
    \begin{bmatrix} w(k) & v(k)  \end{bmatrix}  -  \begin{bmatrix} - \rho_1(k)  \mathcal C^+ u_2(k) &  \rho_2(k)  \mathcal C^- u_1(k) \end{bmatrix} = \begin{bmatrix} 0 & -\rho_2(k) \end{bmatrix}.
\end{align}
While this heuristic is indeed valid, it is a difficult aspect of the problem to track, especially when poles $z_j^\pm$ are accounted for.  The beauty of the method in \cite{Trogdon2013} is that all of this is detected automatically \tcr{because GMRES chooses the oscillatory factors (i.e., it chooses $\alpha$ in $R_{j,\alpha}$) that appear in the basis for the user and, through orthogonality, it removes the oscillatory factors that are not necessary to represent the solution.}

\subsection{Accounting for poles}
\newcommand{\disc}{\mathrm{d}}
We demonstrate how poles are accounted for in the solution of Riemann--Hilbert Problem~\ref{rhp:1}.  The approach is easily adapted to Riemann--Hilbert Problem~\ref{rhp:2}.  If we suppose there is no reflection, i.e., $\rho_1 = \rho_2 = 0$, then we have the following.
\begin{RHproblem}\label{rhp:3}
Find $u_j^\pm(x) \in \mathbb C^{2 \times 2}$, $j = 1,2,\ldots,n_\pm$, such that $$ M_{\disc}(k;x) = I + \sum_{j=1}^{n_+} \frac{u_j^+(x)}{k - z_j^+} + \sum_{j=1}^{n_-} \frac{u_j^-(x)}{k - z_j^-},$$ satisfies
\begin{align*}
    \mathrm{Res}_{k = z_j^+} M_{\disc}(k;x) &= \lim_{k \to z_j^+} M_{\disc}(k;x) \begin{bmatrix} 0 & 0 \\
    c^+_j \E^{2 \I z_j^+ x} & 0 \end{bmatrix},\\
    \mathrm{Res}_{k = z_j^-} M_{\disc}(k;x) &= \lim_{k \to z_j^-} M_{\disc}(k;x) \begin{bmatrix} 0 & c_j^- \E^{-2 \I z_j^- x} \\
    0 & 0 \end{bmatrix}.
\end{align*}
\end{RHproblem}
The matrices $u_j^\pm(x) \in \mathbb C^{2\times 2}$ can be obtained row-by-row. The residue conditions imply that
\begin{align*}
    u_j^+(x) = \begin{bmatrix}  u_{j,1}^+(x) & 0 \\
    u_{j,2}^+(x) & 0 \end{bmatrix}, \quad u_j^-(x) = \begin{bmatrix}  0 & u_{j,1}^-(x) \\ 
    0 & u_{j,2}^-(x)\end{bmatrix}.
\end{align*}
Consider the matrix $Z$ defined entry-wise by
\begin{align*}
    Z_{jk} =  \frac{1}{z_j^+ - z_k^-}.
\end{align*}
Then construct the diagonal matrices
\begin{align*}
C_+ &= \mathrm{diag}(\vec{c}_+), \quad C_- = \mathrm{diag}(\vec{c}_-),\\
    \vec{c}_+ &= \begin{bmatrix} c_1^+\E^{2 \I z_1^+ x} \\ \vdots \\ c_{n^+}^+\E^{2 \I z_{n^+}^+ x}\end{bmatrix}, \quad \vec{c}_- =  \begin{bmatrix} c_1^-\E^{-2 \I z_1^- x}\\ \vdots \\ c^-_{n^-}\E^{-2 \I z_{n_-}^- x} \end{bmatrix}.
\end{align*}
The residue conditions imply
\begin{align*}
    \left[ \begin{array}{c|c} I & -C_+ Z\\\cline{1-2}
    C_-  Z^T & I \end{array}\right] \left[ \begin{array}{c|c} u_{1,1}^+(x) &  u_{1,2}^+(x)\\ \vdots & \vdots \\ u_{n^+,1}^+(x) & u_{n^+,2}^+(x)\\\cline{1-2}  u_{1,1}^-(x) & u_{1,2}^-(x)\\ \vdots & \vdots \\  u_{n^-,1}^-(x) & u_{n^-,2}^-(x) \end{array} \right] = \left[ \begin{array}{c|c} 0 & {\vec c}_+ \\\cline{1-2} 
    {\vec c}_- & 0 \end{array} \right].
\end{align*}
This reduces solving for $M_{\disc}(k;x)$ to pure linear algebra.  And if we only use this for $x \geq 0$, all the exponents involving $x$ that we encounter are decaying as $x$ increases. 

Now define $M_0(k;x) = M(k;x) M_{\disc}(k;x)^{-1}$.  It then follows that $M_0(k;x)$ solves the following Riemann--Hilbert problem without residue conditions.
\begin{RHproblem}\label{rhp:4}
Find $u_0(\cdot;x): \mathbb R \to \mathbb C^{2\times 2}$, such that $$ M_0(k;x) = I + \frac{1}{2\pi \I} \int_{-\infty}^\infty \frac{u_0(k';x)}{k' -k} \D k',$$
\begin{align*}
    M_0^+(k;x) &= M_0^-(k;x) M_{\disc}(k;x)\begin{bmatrix} 1 - \rho_1(k) \rho_2(k) & - \rho_2(k) \E^{-2 \I k x} \\
    \rho_1(k) \E^{2 \I k x} & 1\end{bmatrix}M_{\disc}(k;x)^{-1},\\
    M_0^\pm(k;x) &= \lim_{\epsilon \to 0^+} M_0(k \pm \I \epsilon;x).
\end{align*}
\end{RHproblem}
The associated recovery formula is
\begin{align*}
    \lim_{|k| \to \infty} \I k [\sigma_3, M_0(k;x) + M_{\disc}(k;x)] = \begin{bmatrix} 0 & r(x) \\ q(x) & 0 \end{bmatrix}.
\end{align*}

This problem presents a challenge because the analysis that leads to the conclusion \eqref{eq:first-row-mod} is more more involved.  We use the factorization
\begin{align*}
    M_{\disc}(k;x)&\begin{bmatrix} 1 - \rho_1(k) \rho_2(k) & - \rho_2(k) \E^{-2 \I k x} \\
    \rho_1(k) \E^{2 \I k x} & 1\end{bmatrix}M_{\disc}(k;x)^{-1} \\
    &= U_{\disc}(k;x) L_{\disc}(k;x)^{-1}.
\end{align*}
where
\begin{align*}
    U_{\disc}(k;x) &= M_{\disc}(k;x)\begin{bmatrix} 1  & - \rho_2(k) \E^{-2 \I k x} \\
    0 & 1\end{bmatrix}M_{\disc}(k;x)^{-1},\\
    L_{\disc}(k;x) &= M_{\disc}(k;x)\begin{bmatrix} 1  & 0 \\
    -\rho_1(k) \E^{2 \I k x} & 1\end{bmatrix}M_{\disc}(k;x)^{-1}.
\end{align*}
The singular integral equation for $u_0$ now takes the form
\begin{align*}
    \mathcal C^+ u_0(k) L_{\disc}(k;x) - \mathcal C^{+} u_0(k) U_{\disc}(k;x) = L_{\disc}(k;x) - U_{\disc}(k;x).
\end{align*}
Numerical experiments indicate that the while first row of $u_0$ is not of the form \eqref{eq:form}, it is of the form
\begin{align*}
    \begin{bmatrix} u(k) \E^{-2\I k x} + v(k) \E^{2\I k x} & w(k) \E^{-2\I k x} + z(k) \E^{2\I k x} \end{bmatrix},
\end{align*}
for non-oscillatory functions $u,v,w,z$.  It is not \emph{a priori} clear other exponents such as $\E^{\pm 4\I k x},\E^{\pm 8\I k x}$ should not appear in the solution!  Nevertheless, the algorithm proposed computes an accurate approximate solution of this form where the orthogonality imposed within the GMRES iteration eliminates other exponentials.  But also this seems to indicate that the decomposition $U_{\disc}(k;x) L_{\disc}(k;x)^{-1}$ could be modified to increase efficiency.  

Lastly, we restate a result from \cite{Trogdon2013} that provides a mechanism to perform error analysis.  It is especially important because, in general, the numerical method will not produce a solution that is integrable, but principle value integrals will always exist.
\begin{theorem}\label{t:error}
Let $\mathcal S(\rho_1,\rho_2)$ be a singular integral operator on the left-hand side of \eqref{eq:SIE} and suppose it is invertible on $L^2(\mathbb R)$.  For $0 < \epsilon < 1/\|\mathcal S^{-1}(\rho_1,\rho_2)\|$ suppose $\max_{j=1,2}\|\rho_j - \tilde \rho_j \|_{L^2 \cap L^\infty (\mathbb R)}< \epsilon$ then
\begin{align*}
\|\mathcal S(\rho_1,\rho_2)^{-1} - \mathcal S(\tilde \rho_1,\tilde \rho_2)^{-1}\| \leq \epsilon \frac{\|\mathcal S(\rho_1,\rho_2)^{-1}\|^2}{1 - \epsilon \|\mathcal S(\rho_1,\rho_2)^{-1}\|} =: C(\epsilon).
\end{align*}
\begin{itemize}
\item Let $\tilde u$ be the solution of singular integral equation with $\rho_1,\rho_2$ replaced with $\tilde \rho_1,\tilde \rho_2$. Then 
\begin{align*}
\|u -\tilde u\|_{L^2(\mathbb R)} \leq C(\epsilon) \max_{j=1,2}\|\rho_{j}\|_{L^2(\mathbb R)} + \epsilon \|\mathcal S(\rho_1,\rho_2)\| =: B(\epsilon).
\end{align*}
\item If, in addition, 
\begin{align*}
 \max_{j=1,2}\left|\mathrm{p.v.} \int \left( \rho_j(k) - \tilde \rho_j(k) \right) \D k \right| < \epsilon
\end{align*}
then
\begin{align*}
\left|\mathrm{p.v.} \int \left( u(k) - \tilde u(k) \right) \D k \right| \leq 2 \epsilon + \epsilon \|u\|_{L^2(\mathbb R)} + 2 B(\epsilon) \max_{j=1,2}\|\rho_j\|_{L^2(\mathbb R)}.
\end{align*}
\end{itemize}
\end{theorem}
This theorem can be extended to include $M_{\disc}$ but we leave a full error analysis for future work.

\subsection{Examples} 

\subsubsection{Modulated $\mathrm{sech}$ potential with no discrete spectrum}

In \eqref{eq:sechpot} we choose $\tcr{\amp} = 1.55$, $\gamma = 2.0$. This produces no discrete spectrum $z_j^\pm$.   The magnitude of the coefficients in the expansion of $\rho_1,\rho_2$ in the basis $\{R_{j,0}\}$ with $\nu = 1$ are displayed in Figure~\ref{f:nosol-coefs}.  \tcr{To optimize this, one can follow the approach of \cite{Weideman1994} and choose $\nu$ to minimize the number of coefficients one needs to calculate to achieve a prescribed tolerance.}

\begin{figure}[tbp]
    \centering
    \includegraphics[width=0.8\linewidth]{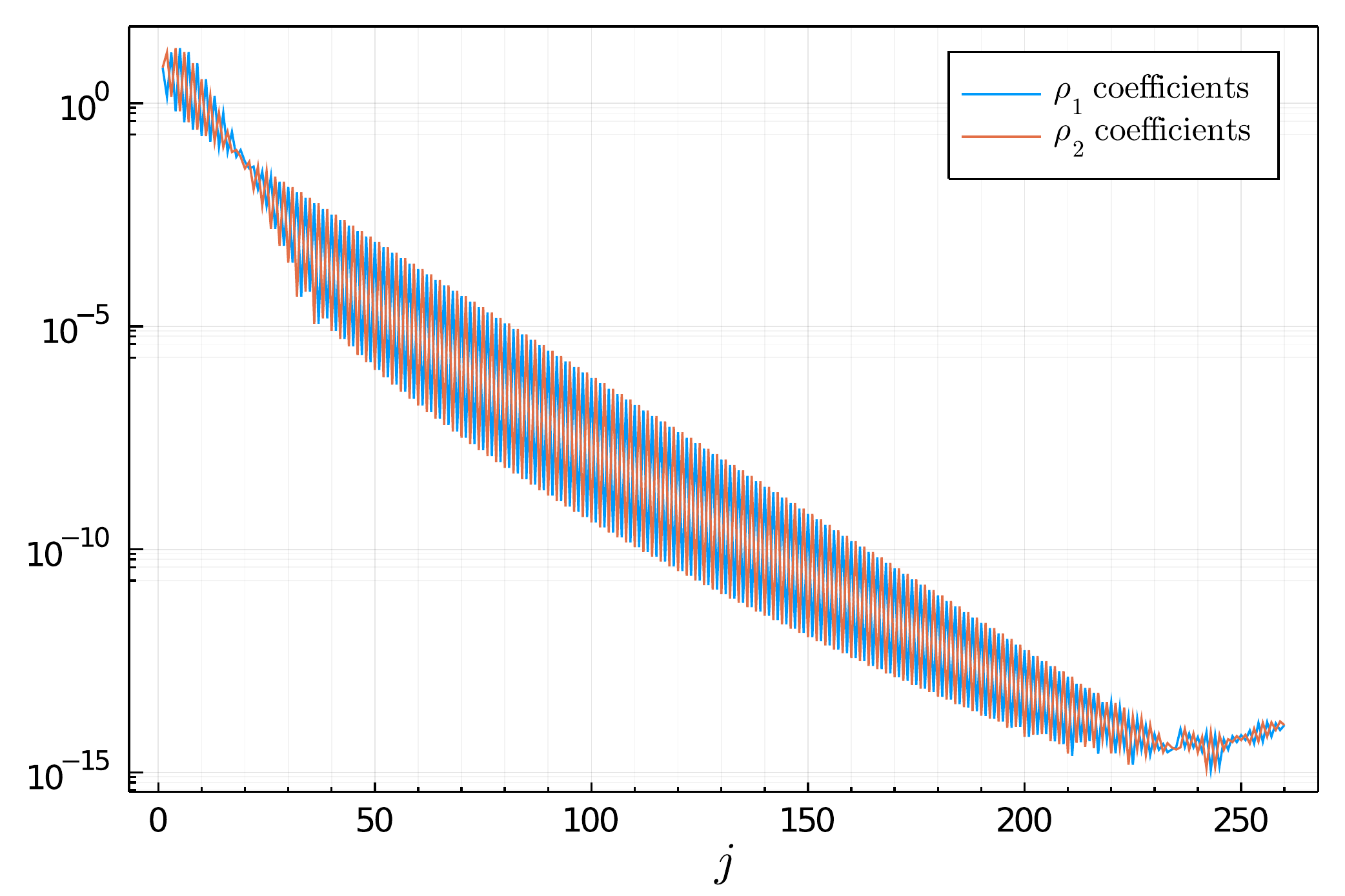}
    \caption{The magnitude of the interlaced expansion coefficients for the reflection coefficients $\rho_1,\rho_2$ for \eqref{eq:sechpot} when $r(x) = - \overline{q(x)}$, $\tcr{\amp} = 1.65, \gamma = 0.1$, using the interpolation operator $\mathcal R_n$ as described in \ref{sec:ratapprox}.  It takes less than 250 coefficient to represent both $\rho_1,\rho_2$ to within approximately machine epsilon.}
    \label{f:nosol-coefs}
\end{figure}

We display the number of GMRES iterations required to achieve a residual less that $2 \times 10^{-12}$ in Figures~\ref{f:FNLS-GMRES-nosol-large-x} and \ref{f:FNLS-GMRES-nosol-small-x}. Importantly, the number of required iterations of GMRES is both very small and decays as $|x|$ increases.
\begin{figure}[tbp]
    \centering
    \begin{subfigure}[b]{.49\textwidth}
         \centering
         \includegraphics[width=\linewidth]{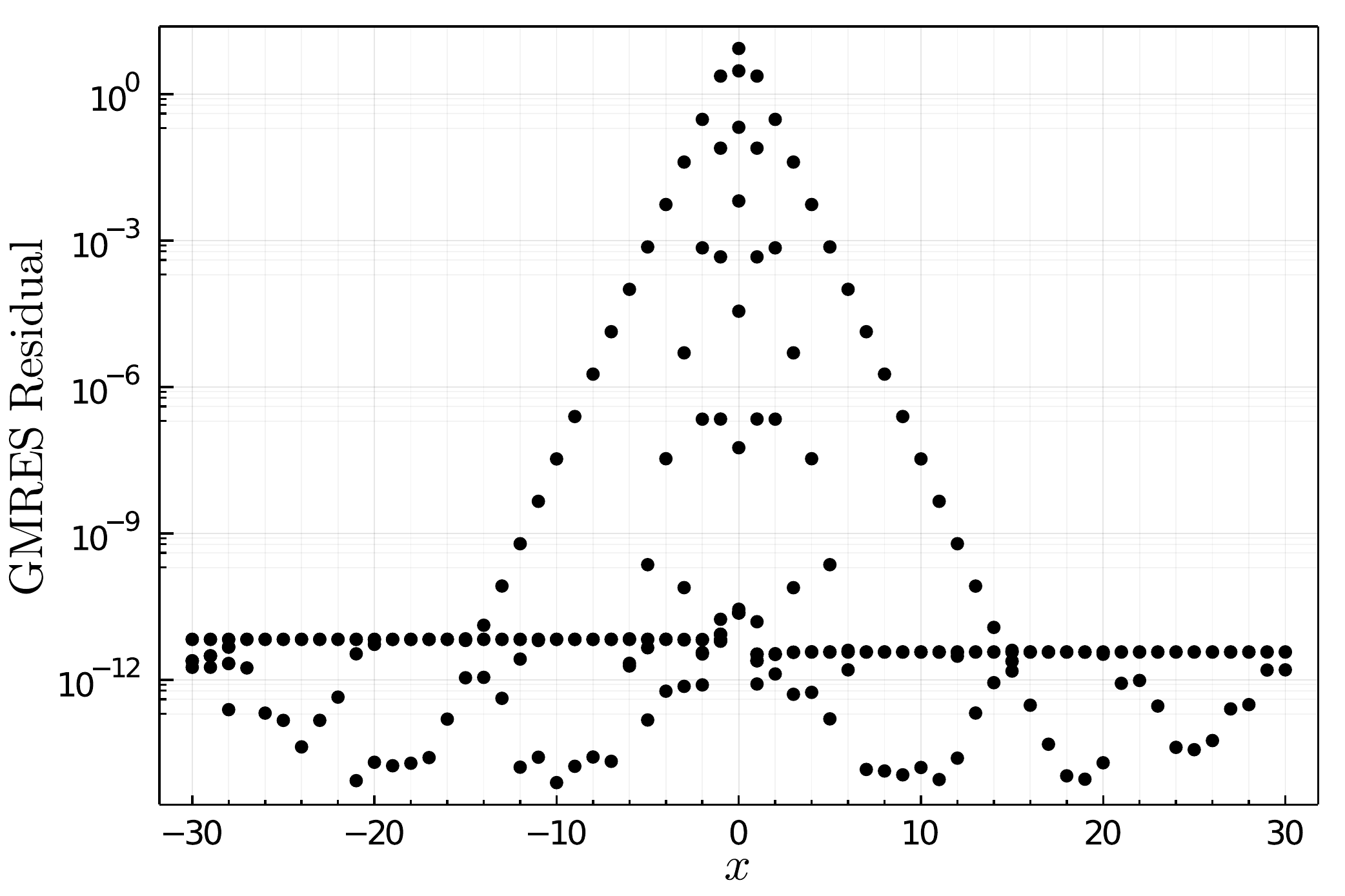}
         \caption{}
         \label{f:FNLS-GMRES-nosol-large-x}
     \end{subfigure}
     \begin{subfigure}[b]{.49\textwidth}
         \centering
         \includegraphics[width=\linewidth]{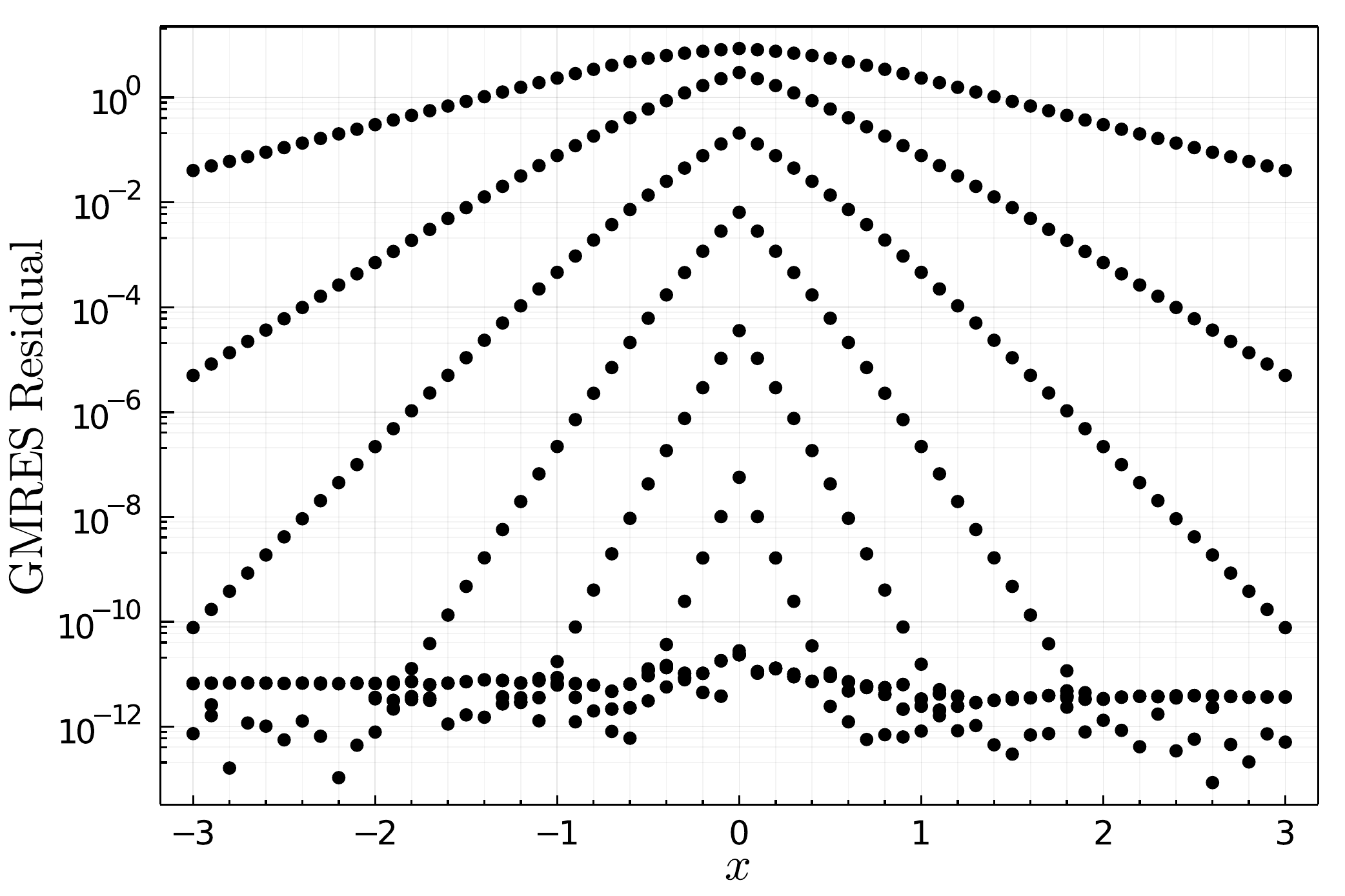}
         \caption{}
         \label{f:FNLS-GMRES-nosol-small-x}
       \end{subfigure}
       \caption{The number of iterations of the GMRES algorithm requires to acheive a residual less than $2 \times 10^{-12}$ for \eqref{eq:sechpot} with $\tcr{\amp} = 1.65, \gamma = 0.1$. For the given value of $x$ we plot the error that is realized after $k = 2,3,4,\ldots$ iterations.  For example, panel (b) shows that for $x = -1$, the $L^2(\mathbb R)$ norm of the second residual is larger than one, the fourth residual is less than $10^{-2}$ and larger than $10^{-4}$. (a) Larger values of $x$. (b) Small values of $x$. }
\end{figure}

\subsubsection{Modulated $\mathrm{sech}$ potential data with discrete spectrum}

Now, in \eqref{eq:sechpot} we choose $\tcr{\amp} = 1.55$, $\gamma = 0.1$. This produces eigenvalues:
\begin{align*}
    z_1^\pm &\approx \pm 1.14793620932364\I,\\
    z_2^\pm &\approx \pm 0.14793620932364\I.  
\end{align*}
These values, while computed numerically, agree with those given in \eqref{eq:zjs} to all displayed digits.  The coefficients in the expansion of $\rho_1,\rho_2$ in the basis $\{R_{j,0}\}$ with $\nu = 1$ are displayed in Figure~\ref{f:sol-coefs}.

\begin{figure}[tbp]
    \centering
    \includegraphics[width=0.8\linewidth]{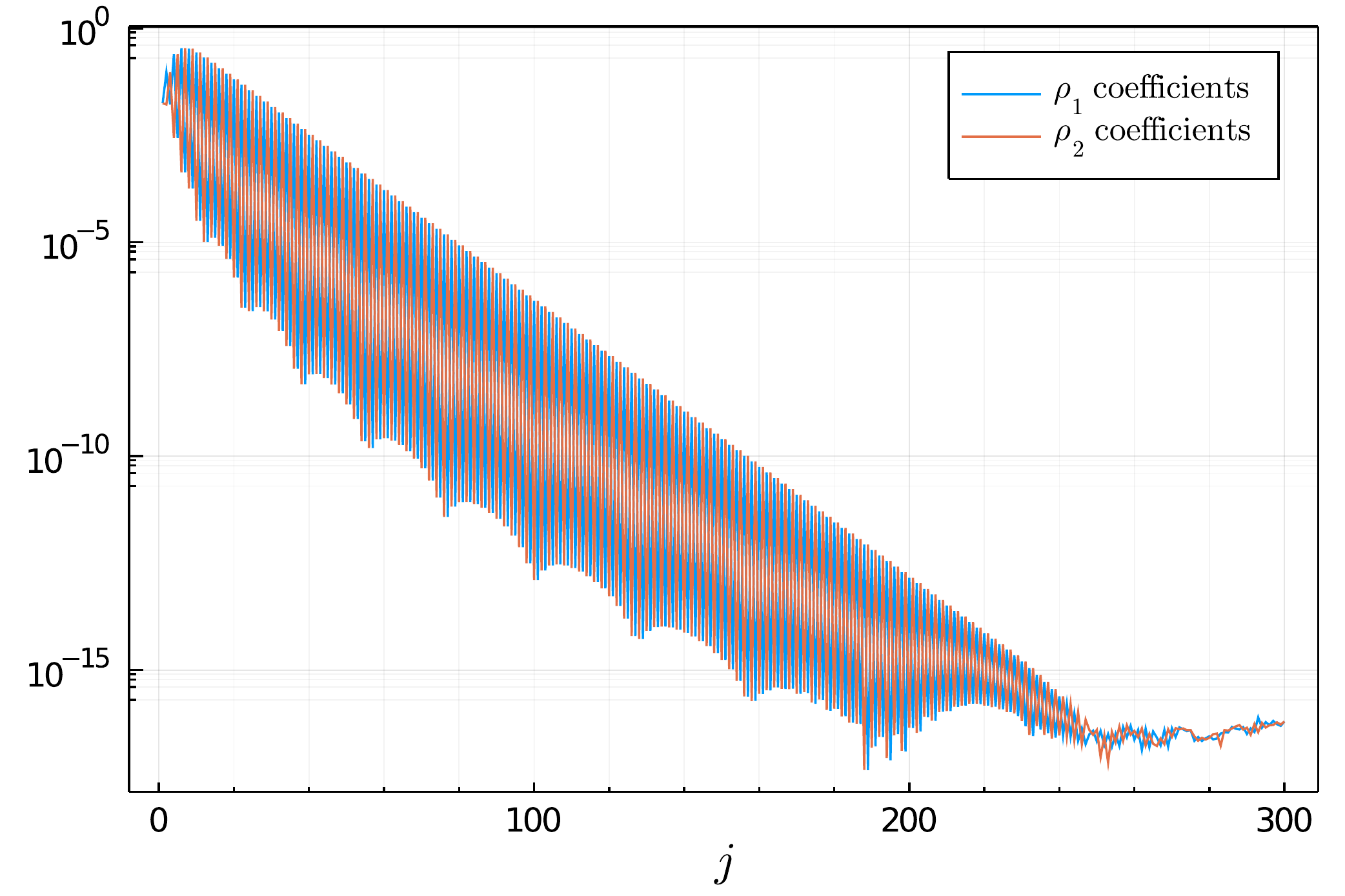}
    \caption{The magnitude of the expansion coefficients $c_j$ for the of the reflection coefficients $\rho_1,\rho_2$ for \eqref{eq:sechpot} when $r(x) = - \overline{q(x)}$, $\tcr{\amp} = 1.55, \gamma = 2.0$, using the interpolation operator $\mathcal R_n$ as described in \ref{sec:ratapprox}.  It takes less than 300 coefficients to represent both $\rho_1,\rho_2$ to within approximately machine epsilon.}
    \label{f:sol-coefs}
\end{figure}

We display the number of GMRES iterations required to achieve a residual less that $2 \times 10^{-12}$ in Figures~\ref{f:FNLS-GMRES-large-x} and \ref{f:FNLS-GMRES-small-x}.  And, as before, the number required iterations decrease as $|x|$ increases.
\begin{figure}[tbp]
    \centering
    \begin{subfigure}[b]{.49\textwidth}
         \centering
         \includegraphics[width=\linewidth]{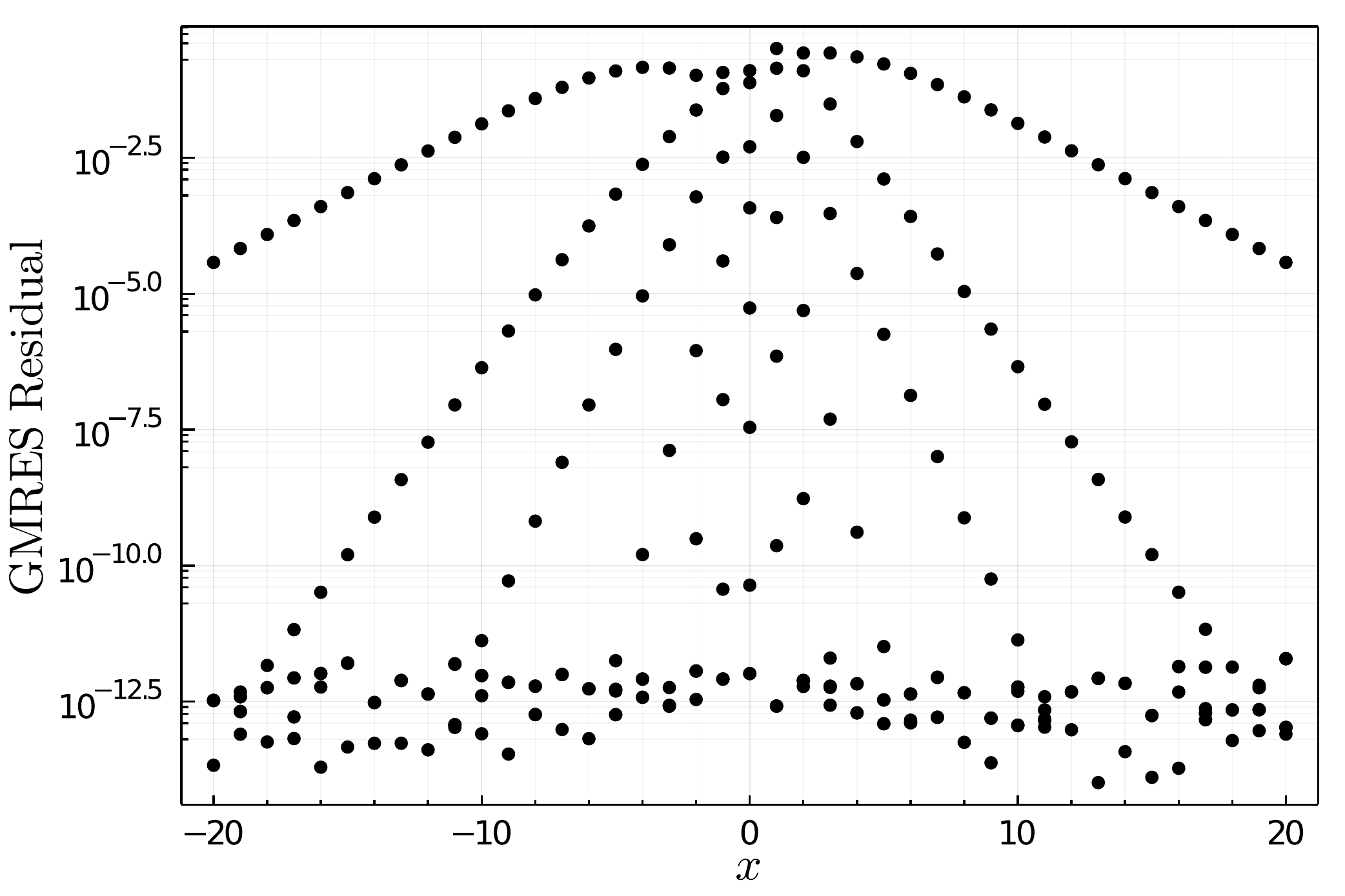}
         \caption{}
         \label{f:FNLS-GMRES-large-x}
     \end{subfigure}
     \begin{subfigure}[b]{.49\textwidth}
         \centering
         \includegraphics[width=\linewidth]{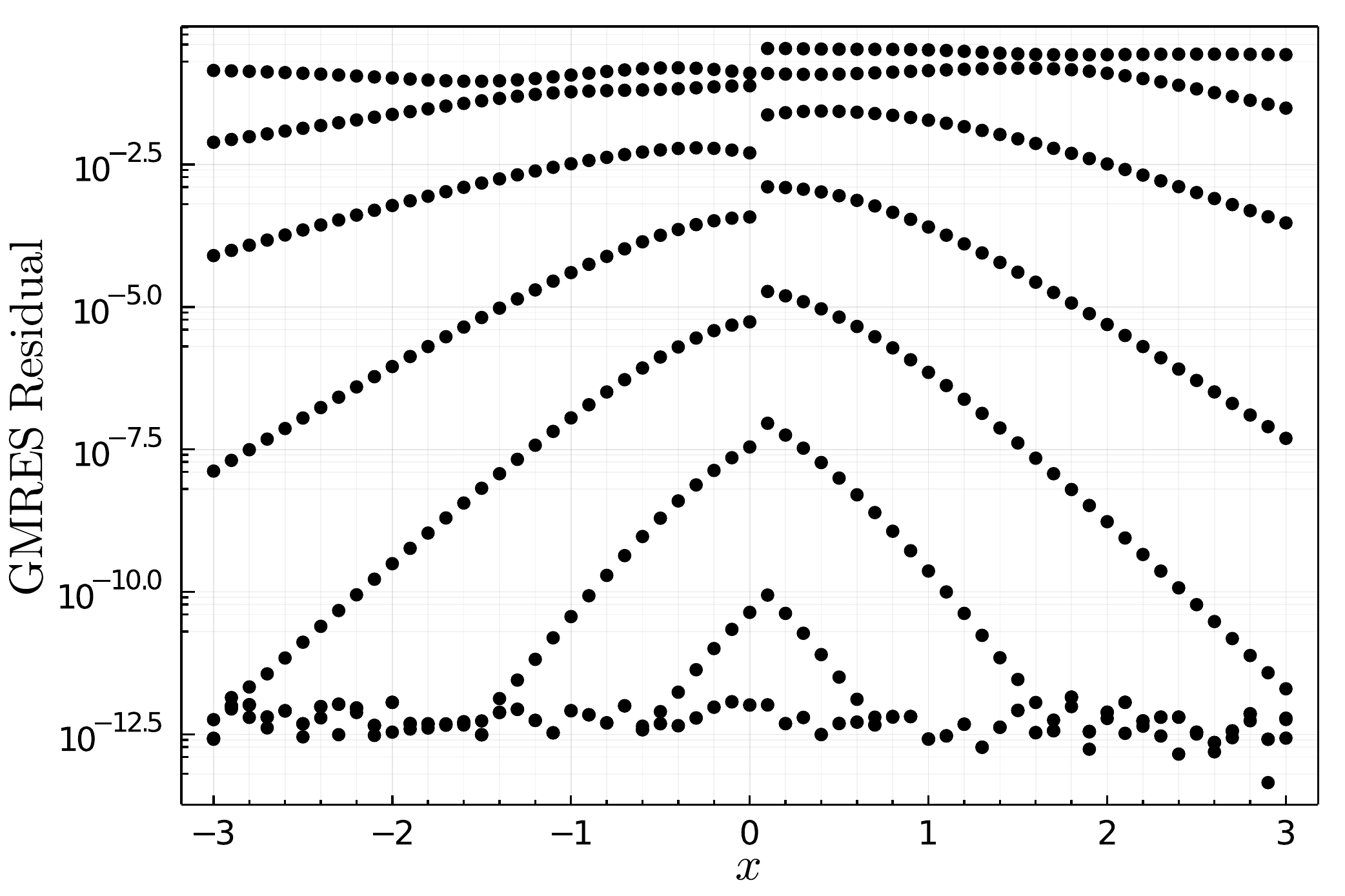}
         \caption{}
         \label{f:FNLS-GMRES-small-x}
       \end{subfigure}
       \caption{The number of iterations of the GMRES algorithm requires to achieve a residual less than $2 \times 10^{-12}$ for \eqref{eq:sechpot} with $\tcr{\amp} = 1.55, \gamma = 2.0$. For the given value of $x$ we plot the error that is realized after $k = 2,3,4,\ldots$ iterations. (a) Larger values of $x$. (b) Small values of $x$.}
\end{figure}

\section{High-precision computations}\label{sec:hi-prec}

Beyond simple arithmetic operations, the numerical method presented in this paper for the inverse problem relies on only two existing features in a software package:
\begin{enumerate}
    \item The fast Fourier transform.
    \item The exponential function.
\end{enumerate}
So, provided these two components can be computed with high or variable precision, and arithmetic is supported in high or variable precision, the numerical method immediately extends to high or variable precision.   This is true of the programming language {\tt Julia}.  The limiting factor in implementing this methodology in high precision is typically computing $\rho_1,\rho_2$ and $\gamma_1,\gamma_2$ with sufficient accuracy.  While we could consider the modulated $\mathrm{sech}$ potential discussed above, we consider another case of interest where
\begin{align*}
    a(k) &= \frac{\Gamma(\tilde a(k,U_0))\Gamma(\tilde b(k,U_0)) }{\Gamma(\tilde c(k))\Gamma(\tilde a(k,U_0) + \tilde b(k,U_0) - \tilde c(k))}\\
    \rho_1(k) &= \frac{a(k) \Gamma(\tilde c(k)) \Gamma(\tilde c(k) -\tilde a(k,U_0) -\tilde b(k,U_0) )}{\Gamma(\tilde c(k)-\tilde a(k,U_0)) \Gamma(\tilde c(k) - \tilde b(k,U_0))},\\
    \tilde b(k,U_0) &= \frac 1 2 - \I k - \left(U_0 - \frac 1 4\right)^{1/2},\\
    \tilde a(k,U_0) &= \frac 1 2 - \I k - \left(U_0 - \frac 1 4\right)^{1/2}, \quad \tilde c(k) = 1 - \I k,\\
    \rho_2(k) & = \overline{\rho_1(\bar k)}.
\end{align*}
This formula gives the reflection coefficient in the case that $r(x) = U_0 \mathrm{sech}^2(x)$ and $q(x) = -1$ \cite{johnson}, the scattering problem associated to the Korteweg-de Vries equation with initial condition $r(x)$.  If $U_0 < 0$ then $a(k)$ has no zeros in the upper-half plane and, for the sake of simplicity, this is the only case we consider in this section.

Due to the fact that $q(x)$ does not decay at infinity, the reconstruction formula for $r(x)$ has to be modified \cite{AblowitzClarksonSolitons} to
\begin{align*}
    r(x) = \lim_{k \to \infty} 2\I k \frac{\partial}{\partial x}\left[ M_{21}(k;x) + M_{11}(k;x)  \right].
\end{align*}
This is best accomplished by solving for a vector-valued unknown $u$ and replacing the right-hand side in \eqref{eq:SIE} with
\begin{align*}
    \begin{bmatrix} \rho_1(k) \E^{2 \I k x} & -\rho_2(k) \E^{-2 \I k x}\end{bmatrix}.
\end{align*}
Once this is solved for an approximation of $u$, the resulting singular integral equation can easily be differentiated with respect to $x$ and solved again with a new right-hand side to approximate $ \frac{\partial u}{\partial x}$.

With the explicit formulae in hand for $\rho_1,\rho_2$, we use {\tt Mathematica} to compute $\rho_1$ at a set of grid points that are sufficient to then create expansions $\sum_{j=1}^n r_j R_{j,0}(k)$, $\sum_{j=1}^n \hat r_j R_{j,0}(k)$ such that
\begin{align*}
   \sup_{k \in \mathbb R} \left| \rho_1(k) - \sum_{j=-n}^n r_j R_{j,0}(k)\right| &< 10^{-50},\\
   \sup_{k \in \mathbb R} \left| k\rho_1(k) - \sum_{j=-n}^n \hat r_j R_{j,0}(k)\right| &< 10^{-50}.
\end{align*}
We find that using $2048$ terms in the sum is sufficient.  The approximation of $k \rho_1(k)$ is needed when the derivative $\frac{\partial u}{\partial x}$ is computed.

Replacing $\rho_1(k)$ with this approximation, GMRES until the residual is less than $10^{-70}$ to get approximations of $u$ and $ \frac{\partial u}{\partial x}$.  Then we find the reconstruction of $r(x)$ as
\begin{align*}
 r(x) \approx \frac{1}{\pi} \mathrm{p.v.} \int_{-\infty}^\infty \frac{\partial u_1}{\partial x}(k;x) \D k.
\end{align*}
All of this implies an error of approximately $10^{-50}$ as is confirmed by numerical experiments, see Figure~\ref{f:hiprec-gmres}.
\begin{figure}[tbp]
    \centering
    \begin{subfigure}[b]{.49\textwidth}
         \centering
         \includegraphics[width=\linewidth]{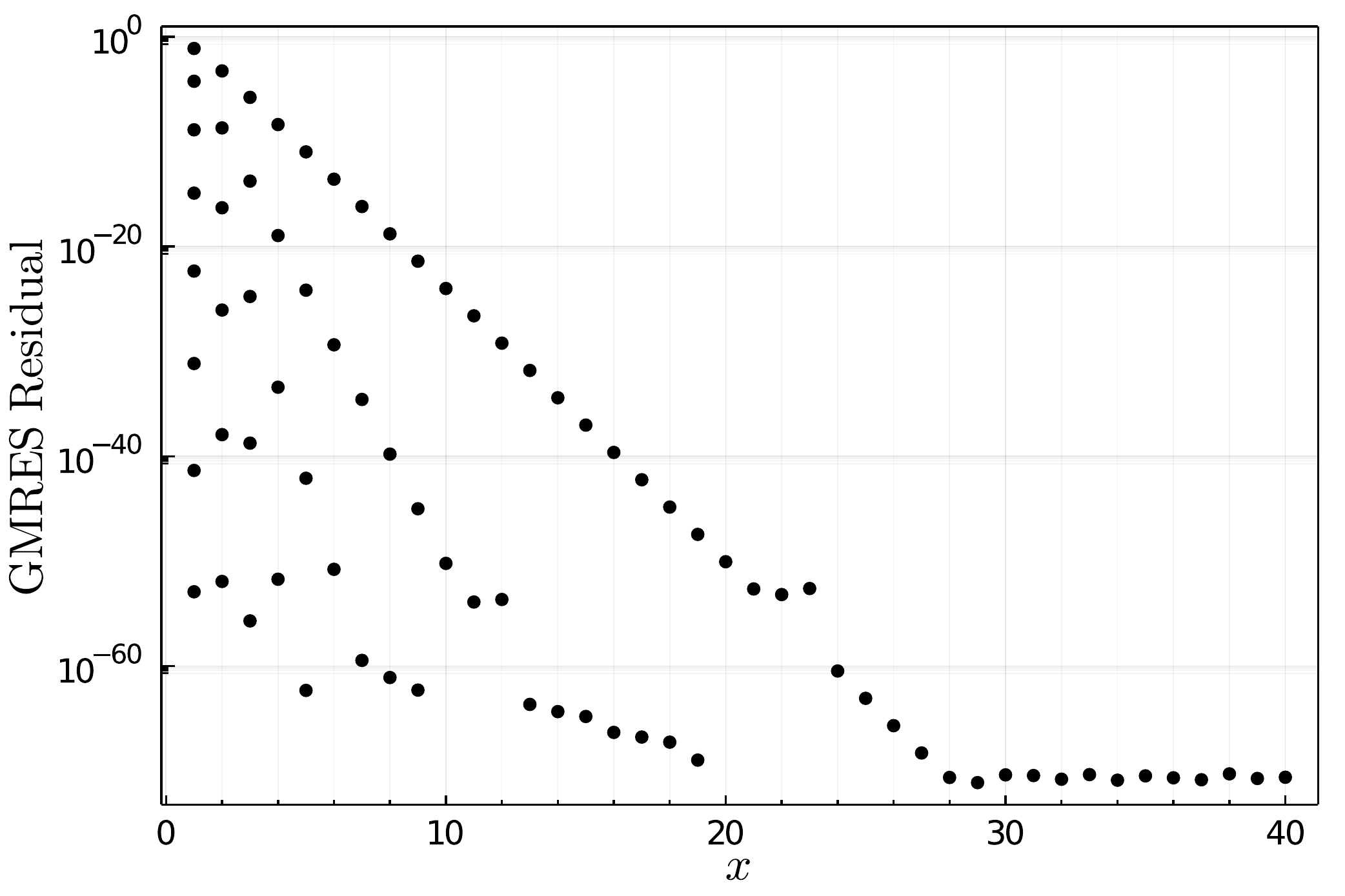}
         \caption{}
         \label{f:hiprec-gmres}
     \end{subfigure}
     \begin{subfigure}[b]{.49\textwidth}
         \centering
         \includegraphics[width=\linewidth]{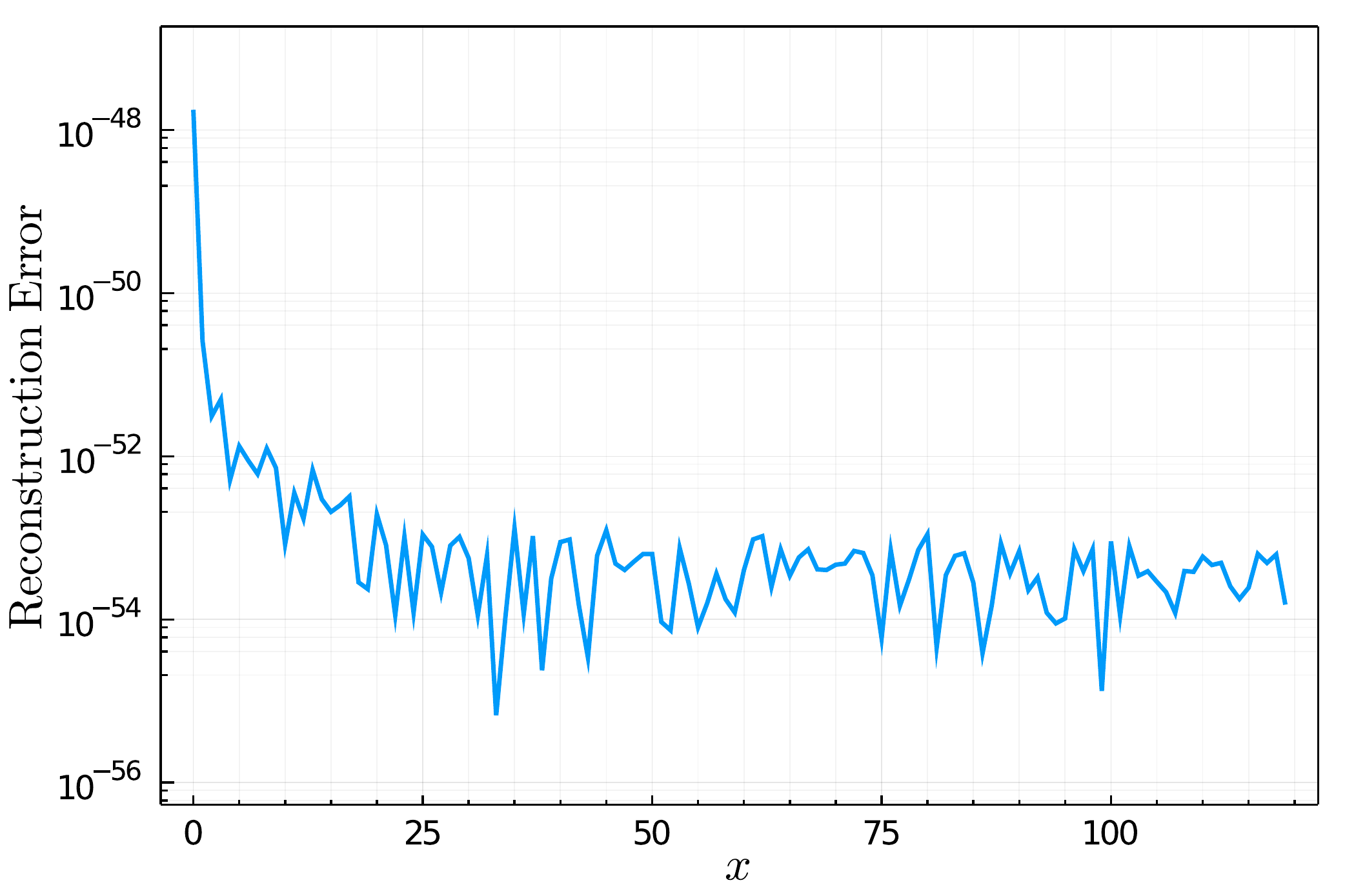}
         \caption{}
         \label{f:hiprec-err}
       \end{subfigure}
       \caption{(a) The number of iterations of the GMRES algorithm required to acheive a residual less than $10^{-60}$ for $r(x) = U_0 \mathrm{sech}^2(x)$, $q(x) = -1$ with $U_0 = -1$. For the given value of $x$ we plot the error that is realized after $k = 2,3,4,\ldots$ iterations.  For example, panel (a) shows that for $x = 10$, the $L^2(\mathbb R)$ norm of the second residual is larger than $10^{-20}$ and the third residual is less than $10^{-40}$ and larger than $10^{-60}$. (b) The error in the recovery of $r(x)$ as $x$ increases.  }
\end{figure}

\section{Conclusions and open questions}\label{sec:conc}

We have described an approach to scattering and inverse scattering for the AKNS system using only rational functions and the fast Fourier transform.  The approach can be seen as a nonlinearization of two methods for the Fourier transform on $\mathbb R$.  The forward scattering transform is effective but suffers from some loss of accuracy for exceptional values of $k$.  The method for inverse transform requires no contour deformation and is more efficient for larger values of $|x|$.  A cursory implementation of these ideas can be found here \cite{Trogdon2021,Trogdon2021a}.

But there are many open questions and possible directions for improvement of the method:

\begin{itemize}
    \item It is important to understand how choice of $\mathfrak r$ versus $\mathfrak g$ (or another choice altogether) affects the convergence rate and the overall accuracy of the method.  Additionally, it is important to understand if the parameter $\nu$ that is present in the basis $\{R_{j,0}\}$ can be chosen optimally.
    \item As mentioned, the forward scattering method described here exhibits some loss of accuracy.  It is possible that appropriate use the the adaptive QR algorithm may eleviate this issue.
    \item The formulae given in \eqref{eq:rja} for the oscillatory Cauchy operator should produce effective bounds on the inverse operator for the operators on the left-hand side of \eqref{eq:SIE} and \eqref{eq:SIE2}.  These can be incorporated along with Theorem~\ref{t:error} to produce error bounds.  And then, giving a proof of the convergence rate of the GMRES algorithm would complete the analysis.
    \item It is of interest to be able to compute scattering data in high-precision.  The operator in \eqref{eq:scatteringsys} can be applied to a vector using the FFT.  If a sparse preconditioner can be devised, and maybe implemented using the adaptive QR algorithm, then GMRES, or some other iterative method, might provide an effective way to treat thousands of unknowns using high-precision arithmetic. 
\end{itemize}

\section*{Acknowledgements}

The author would like to thank Bernard Deconinck and Sergey Dyachenko for suggesting the problem of using high-precision arithmetic in computing the inverse scattering transform.  This material is based upon work supported by the National Science Foundation under Grant DMS-1945652.

\appendix

\section{Rational functions and oscillatory Cauchy integrals}\label{sec:osccauchy}

\subsection{Rational approximation}\label{sec:ratapprox}
\renewcommand{\omega}{k}

We consider the problem of the rational approximation of $f : \mathbb R \to \mathbb C$, under suitable regularity conditions.  We follow \cite{Trogdon2014} and first discuss trigonometric interpolation of an associated continuous periodic function $F$. For $n \in \mathbb N$, define $\theta_j = 2 \pi j / n$ for $j = 0, \ldots, n-1$ and two positive integers
\begin{align*}
n_+ = \lfloor n/2 \rfloor, \quad n_- = \lfloor (n-1)/2 \rfloor.
\end{align*}
Note that $n_+ + n_- + 1 = n$ regardless of whether $n$ is even or odd.

\begin{definition}
The discrete Fourier transform of order $n$ of a continuous function $F$ is the mapping
\begin{align*}
\mathcal F_n F &= [ \tilde F_0, \tilde F_1, \ldots, \tilde F_n]^T,\\
\tilde F_k &= \frac{1}{n} \sum_{j=0}^{n-1} e^{-ik\theta_j} F(\theta_j).
\end{align*}
\end{definition}
The fast Fourier transform (FFT) \cite{Cooley1965} is an algorithm that implements the discrete Fourier transform in $O(n \log n)$ floating point operations.  Note that the dependence on $n$ is implicit in the $\tilde F_k$ notation.  This formula produces the coefficients for the trigonometric interpolant of $F$:
\begin{align*}
  \mathcal I_nF(\theta) = \sum_{k=-n_-}^{n_+} e^{ik\theta} \tilde F_k, \quad \mathcal I_nF(\theta_j) = F(\theta_j), \quad j = 0,\ldots,n-1.
\end{align*}
Now, define a one-parameter family of M\"obius transformations
\begin{align*}
T_\nu(\omega) = \frac{\omega - \I \nu}{\omega+\I \nu}, \quad T^{-1}_\nu(z) = \frac{\nu}{\I} \frac{z+1}{z-1}, \quad \nu > 0.
\end{align*}
And related to these transformations are set of basis functions
\begin{align}\label{eq:Rjs}
    R_{j,\alpha}(\omega) = \E^{\I k \alpha} \left[ \left( \frac{\omega - \I \nu}{\omega + \I \nu} \right)^j  - 1\right].
\end{align}
For $\alpha \neq 0$ we call these functions \emph{oscillatory rational functions}.

For each $\nu$, $T_\nu$ maps the real axis onto the unit circle. Assume $f$ is a smooth function on $\mathbb R$, decaying at infinity.  We further assume that $f$ is either rapidly decaying or a rational function.  Then $f$ is mapped to a smooth function on $[0,2\pi]$ by $F(\theta) = f(T_\nu^{-1}(\E^{\I\theta}))$.  Thus, the FFT may be applied to $F(\theta)$ to obtain an interpolant $\mathcal I_n F(\theta)$. The transformation $x = T_\nu^{-1}(\E^{\I\theta})$ is inverted:
\begin{align*}
\mathcal R_nf(\omega) := \mathcal I_n F(T_\nu^{-1}(\omega)),
\end{align*}
is a rational approximation of $f$.  This approximation will converge rapidly as $n \to \infty$ \cite{Trogdon2014}.  We find
\begin{align*}
\mathcal R_nf(\omega) = \sum_{j=-n_-}^{n_+}\tilde F_j T_\nu(\omega)^j = \sum_{j=-n_-}^{n_+}\tilde F_j R_{j,0}(\omega)
\end{align*}
because $f(\infty) = 0$ and we choose $\theta = 0$ ($\omega = \infty)$ to be an interpolation point.  Finally, define the oscillatory interpolation operator by
\begin{align}\label{eq:interp_op}
\mathcal R_{n,\alpha} f(\omega) = \E^{\I \alpha \omega} \mathcal R_n[f(\cdot) \E^{- \I \alpha (\cdot)} ](\omega).
\end{align}
Note that this operator also depends on $\nu$ but we suppress that dependency.

We overload this notation to allow inputs that are vectors.  For $\vec c = [c_0,\ldots,c_{n-1}]^T \in \mathbb R^n$ define
\begin{align*}
    \mathcal R_{n,\alpha} \vec c(\omega) = \mathcal R_{n,\alpha} g(\omega),
\end{align*}
where $g$ is any function on $\mathbb R$ satisfying $g(T_\nu^{-1}(\E^{\I \theta_j})) = c_j$ for $j = 0,1,\ldots,n-1$.  This allows the interpolation operator to just take in data at the interpolation points and return the interpolant.

In the next two sections we discuss important properties of $R_{j,\alpha}(\omega)$.

\subsection{Differentiation and multiplication}\label{sec:diffint}

A straightforward calculation \cite{Trogdon2014} shows that
\begin{align*}
    R_{j,0}'(\omega) = \frac{\I}{\nu} \left[ - \frac{j}{2} R_{j+1,0}(\omega) + j R_{j,0}(\omega) - \frac{j}{2} R_{j-1,0}(\omega) \right],
\end{align*}
so that differentiation is a tridiagonal operator.  This is also a consequence of the fact that, as derived below, the Fourier transform of these functions are weighted orthogonal polynomials \cite{Iserles2018}.  We also have
\begin{align}\label{eq:tridiag-diff}
    R_{j,\alpha}'(\omega) = \frac{\I}{\nu} \left[ - \frac{j}{2} R_{j+1,\alpha}(\omega) + (j + \alpha\nu) R_{j,\alpha}(\omega) - j R_{j-1,\alpha}(\omega) \right].
\end{align}
When it comes to function multiplication we use the fact that
\begin{align*}
    R_{j,\alpha}(\omega) R_{\ell,\beta}(\omega) = R_{j+\ell,\alpha + \beta}(\omega) - R_{j,\alpha + \beta} - R_{\ell,\alpha+\beta}(\omega).
\end{align*}

Suppose $g = \sum_j c_j R_{j,\alpha}(\omega)$ and consider the operator $\mathcal M_g f = g f$ via:
\begin{align*}
    g(\omega) R_{\ell,\beta}(\omega) & = \sum_j c_j R_{j + \ell,\alpha + \beta}(\omega) - \sum_j c_j R_{j,\alpha + \beta}(\omega) - \sum_j c_j R_{\ell,\alpha + \beta}(\omega)\\
    & = \sum_j c_{j-\ell} R_{j,\alpha + \beta}(\omega)- \sum_j c_j R_{j,\alpha + \beta}(\omega)  - \left[ \sum_j c_j \right] R_{\ell,\alpha + \beta}(\omega).
\end{align*}
This implies that $\mathcal M_g$ has a bi-infinite matrix representation as
\begin{align*} 
\mathcal M(\vec c) &:=    \mathcal T(\vec c) - \vec c \begin{bmatrix} \cdots & 1 & 1 & \cdots \end{bmatrix} - \left[ \sum_j c_j \right] I, \\
\vec c &= \begin{bmatrix}
\cdots & c_{-1} & c_0 & c_{1} & \cdots \end{bmatrix}^T. 
\end{align*}
Here $\mathcal T(\vec c)$ is the Toeplitz operator with entry $(i,j)$ given by $c_{i-j}$.  Note that the ordering of $\vec c$ is different here than that used in the main text.  To accommodate this, we introduce the interlacing operator $\mathcal I$ mapping bi-infinite sequences to semi-infinite ones by
\begin{align*}
    \mathcal I \vec c = \begin{bmatrix} c_0 \\ c_1 \\ c_{-1} \\ c_2 \\ \vdots \end{bmatrix}.
\end{align*}
Then the semi-infinite matrix $\mathcal M_{\mathcal I}(\vec c)$, for a semi-infinite vector $\vec c$,  is defined to be the matrix found by deleting the first row and first column of
\begin{align}
    \mathcal I\mathcal M(\mathcal I^{-1}\vec c) \mathcal I^{-1}.\label{eq:multmat}
\end{align}
This row and column is deleted because $R_{0,\alpha}  = 0$.

\subsection{Cauchy integrals}\label{sec:cauchy-app}

Recall the notation for the Cauchy integral
\begin{align}\label{eq:cauchy-def}
    \begin{split}
    \mathcal C f(\omega) = \frac{1}{2 \pi \I} \int_{-\infty}^\infty \frac{f(k')}{k-k'} \D k, \quad \omega \in \mathbb C \setminus \mathbb R,\\
    \quad \mathcal C^\pm f(\omega) = \lim_{\epsilon \to 0^+} \mathcal C f(\omega \pm \I \epsilon), \quad \omega \in \mathbb R.
    \end{split}
\end{align}
Consider the problem of expressing $\mathcal C R_{j,\alpha}(\omega)$ in terms of $R_{j,\alpha}(\omega)$ and $R_{j,0}(\omega)$.  It is not \emph{a priori} clear this is possible.  The following is from \cite{Trogdon2013} and it shows that this is indeed possible.  Define for $j > 0$, $n > 0$,
\begin{align*}
\gamma_{j,n}(\alpha) = - \frac{j}{n} e^{- |\alpha| \nu } \left(\begin{array}{cc} j-1 \\ n \end{array}\right) \phantom{.}_1F_1(n-j,1+n,2 |\alpha| \nu),
\end{align*}
where $_1F_1$ is Krummer's confluent hypergeometric function
\begin{align}\label{eq:krum}
_1F_1(a,b,z) = \sum_{\ell =0}^\infty \frac{\Gamma(a +\ell)}{\Gamma(a)} \frac{\Gamma(b)}{\Gamma(b+\ell)} \frac{z^\ell}{\ell!},
\end{align}
and $\Gamma$ denotes the Gamma function \cite{DLMF}.   Further, define
\begin{align*}
\eta_{j,n}(\alpha) = \sum_{\ell=n}^{j} (-1)^{n+\ell} \left(\begin{array}{cc} \ell \\ n \end{array}\right) \gamma_{\ell,n}(\alpha).
\end{align*}
\begin{theorem}[\hspace{-1pt}\cite{Trogdon2013}]\label{Theorem:CauchyAction}
If $\alpha j \geq 0$ then
\begin{align*}
  \mathcal C^+ R_{j,\alpha}(\omega) &= \begin{cases}
    R_{j,\alpha}(\omega) & j > 0,
    \\ 0 & j < 0, \end{cases}\\
\mathcal C^- R_{j,\alpha}(\omega) &= \begin{cases} 0 & j > 0 \\ -R_{j,\alpha}(\omega) & j < 0.\end{cases}
\end{align*}
If $\alpha j < 0$ then
\begin{align*}
\mathcal C^+ R_{j,\alpha}(\omega) & = \begin{cases} - \displaystyle \sum_{n=1}^{j} \eta_{j,n}(\alpha) R_{n,0}(\omega) & j > 0,\\ 
R_{j,\alpha}(\omega) +  \displaystyle\sum_{n=1}^{-j} \eta_{j,n}(\alpha) R_{-n,0}(\omega) & j < 0, \end{cases}\\
\mathcal C^- R_{j,\alpha}(\omega) & = \begin{cases} -R_{j,\alpha}(\omega)-  \displaystyle\sum_{n=1}^{j} \eta_{j,n}(\alpha) R_{n,0}(\omega) & j > 0,\\ 
  \displaystyle\sum_{n=1}^{-j} \eta_{j,n}(\alpha) R_{-n,0}(\omega) & j < 0. \end{cases}
\end{align*}
\end{theorem}

The another important expression from \cite{Trogdon2013} is
\begin{align}\label{Laguerre}
\hat {R}_j(\omega) = \begin{cases}
0 & \sign(j) = -\sign(\omega), ~~\omega \neq 0,\\
- 2 \pi |j| \nu & \omega = 0,\\
\displaystyle -4 \pi e^{- |\omega| \nu} \nu L_{|j|-1}^{(1)}(2 |\omega|\nu) & \text{otherwise},\end{cases}
\end{align}
where $L^{(\alpha)}_n(x)$ is the generalized Laguerre polynomial of order $n$ \cite{DLMF}.

We perform all of our computations in coefficient space, that is, given the coefficients in the finite sum $f$ of the $R_{j,\alpha}$'s, we want to compute the corresponding coefficients in the corresponding expansion of $\mathcal C^\pm f$.  The above formulae achieve this task, but inefficiently.

But these formulae do provide an important theoretical guide to speed up the computation.  From \cite{Trogdon2014} we define
\begin{align*}
 r_{j,\alpha}\left(  \frac{-2 \I \sigma \nu }{\omega + \sigma \I \nu} \right)&:=\mathrm{Res}_{k' = - \sigma \nu \I }\,\left\{ R_{j,\alpha}(\omega') \frac{1}{\omega' - \omega} \right\} = \sum_{n = 1}^{|j|} \gamma_{j,n} \left( \frac{-2 \I \sigma \nu }{\omega + \sigma \I \nu} \right)^n, \\ \sigma &= \sign(j).
\end{align*}
This is convenient when $\alpha j < 0$ and we can rewrite everything in terms of $r_{j,\alpha}$:
\begin{align}\label{eq:rja}
\begin{split}
  \mathcal C^+ R_{j,\alpha}(\omega) &= \begin{cases}
    -r_{j,\alpha}\left(  \frac{-2 \I \sigma \nu }{\omega + \sigma \I \nu} \right) & j > 0,
    \\ 
    R_{j,\alpha}(\omega) + r_{j,\alpha}\left(  \frac{-2 \I \sigma \nu }{\omega + \sigma \I \nu} \right) & j < 0, \end{cases}\\
\mathcal C^- R_{j,\alpha}(\omega) &= \begin{cases} - R_{j,\alpha}(\omega) - r_{j,\alpha}\left(  \frac{-2 \I \sigma \nu }{\omega + \sigma \I \nu} \right) & j > 0 \\ r_{j,\alpha}\left(  \frac{-2 \I \sigma \nu }{\omega + \sigma \I \nu} \right)& j < 0,\end{cases}
\end{split}
\end{align}
where $\sigma = \sign(j)$ in each line.

Using \eqref{eq:krum} we find
\begin{align*}
\gamma_{j,n}(\alpha) = - \E^{- |\alpha| \nu} L_{|j|-n}^{(n)}( 2 |\alpha| \sigma).
\end{align*}
As a secondary, but important note, it follows that
\begin{align}\label{eq:large-k-rj}
    \lim_{\omega \to \infty} \omega r_{j,\alpha}\left(  \frac{-2 \I \sigma \nu }{\omega + \sigma \I \nu} \right) = -2 \I \sigma \nu \gamma_{j,1}.
\end{align}
This can be used to give a derivation of \eqref{Laguerre}.

We then use the well-known relation for Laguerre polynomials  \cite{DLMF}
\begin{align*}
L_n^{(\alpha)}(x) = L_n^{(\alpha+1)}(x) - L_{n-1}^{(\alpha + 1)}(x),
\end{align*}
to obtain the recurrence relation for $j \geq 1$
\begin{align*}
  r_{j,\alpha}(z) = (1 + z) r_{j-1,\alpha}(z) + z ( L_{j-1}^{(1)}(2 |\omega| \nu) - L_{j-2}^{(1)}(2 |\omega| \nu)),
\end{align*}
where $L_{-1}^{(1)}(x) := 0$, $r_{0,\alpha} := 0$.  Additionally. $L_j^{(1)}(x)$ can be computed by its three-term recurrence relation \cite{DLMF}.  So, given a vector $\vec{\omega}$ of $m$ values for $\omega$, with $\vec z =  \frac{-2 \I \sigma \nu }{\vec \omega + \sigma \I \nu}$, the matrix
\begin{align*}
  M_\sigma(\vec \omega) = \begin{bmatrix} r_{1,\alpha}(\vec z) & r_{2,\alpha}(\vec z) & \cdots & r_{m,\alpha}(\vec z) \end{bmatrix}
\end{align*}
can be constructed in $O(m^2)$ operations building each column from the previous column.  Furthermore, to compute $M_\sigma(\vec k) \vec c$ the matrix $M_\sigma(\vec k)$ never needs to be constructed.  So, for $f_+(\omega) = \sum_{j=1}^m c_j R_{j,\alpha}(\omega)$ with $\vec c = [c_1, \ldots,c_m]^T$ we have
\begin{align*}
  \mathcal C^+ f_+(\vec \omega) &= \begin{cases} f_+(\vec \omega) & \alpha \geq 0,\\
    -M_{+1}(\vec \omega) \vec c  & \alpha < 0,\end{cases}\\
  \mathcal C^-f_+(\vec \omega) &= \begin{cases} 0& \alpha \geq 0,\\
   -f_+(\vec \omega) - M_{+1}(\vec \omega) \vec c  & \alpha < 0.\end{cases}
\end{align*}
Similarly, for $f_-(\omega) = \sum_{j=1}^m c_{-j} R_{-j,\alpha}(\omega)$ with $\vec c = [c_{-1}, \ldots,c_{-m}]^T$ we have
\begin{align*}
  \mathcal C^+ f_-(\vec \omega) &= \begin{cases} f_-(\vec \omega) + M_{-1}(\vec \omega) \vec c & \alpha \geq 0,\\
    0  & \alpha < 0,\end{cases}\\
  \mathcal C^-f_-(\vec \omega) &= \begin{cases} M_{-1}(\vec \omega) \vec c & \alpha \geq 0,\\
   -f_-(\vec \omega)   & \alpha < 0,\end{cases}
\end{align*}
To compute the Cauchy operator one separates the resulting function into an oscillatory function and a non-oscillatory function.  The coefficients in the expansion of the oscillatory function are always found by multiplying the original coefficients by $1,-1$ or $0$.  The computation of the non-oscillatory part is more difficult.  To do this, one evaluates it pointwise on the real axis at the appropriate points for the interpolation operator $\mathcal R_{n,0}$ using the matrix $M_\sigma$.  Then the interpolation operator can be applied to compute the coefficients.  This gives a stable numerical algorithm that bypasses the hypergeometric function expansion defined above.

So, for example, if one wants to compute  $\mathcal C^+ f(\omega)$ where $f(\omega) = \sum_{j=-m}^m c_j R_{j,\alpha}(\omega)$ when $\alpha > 0$, we have
\begin{align}\label{eq:cauchy_better}
    \mathcal C^+ f(\omega) &= \sum_{j=-m}^m c_j R_{j,\alpha}(\omega) + \mathcal R_{2m+1,0}\vec v (\omega),\\
    \vec v &= M_\sigma (\vec \omega) \vec c, \quad \sigma = -1,\notag\\
    \vec \omega &= \begin{bmatrix} T_\nu^{-1}(\E^{\I \theta_0}) & \cdots & T_\nu^{-1}(\E^{\I \theta_{2m}})\end{bmatrix},\notag\\
    \vec c &= \begin{bmatrix} c_{-1} & \cdots & c_{-m}\end{bmatrix}^T.\notag
\end{align}
Or when $\alpha < 0$
\begin{align}\label{eq:cauchy_better2}
    \mathcal C^+ f(\omega) &= -\mathcal R_{2m+1,0}\vec v (\omega),\\
    \vec v &= M_\sigma (\vec \omega) \vec c, \quad \sigma = 1,\notag\\
   \vec \omega &= \begin{bmatrix} T_\nu^{-1}(\E^{\I \theta_0}) & \cdots & T_\nu^{-1}(\E^{\I \theta_{2m}})\end{bmatrix},\notag\\
    \vec c &= \begin{bmatrix} c_{1} & \cdots & c_{m}\end{bmatrix}^T.\notag
\end{align}
This is the case where the Cauchy integral of an oscillatory rational function becomes a non-oscillatory rational function and this operator can be thought of as a sort of smoothing operator.  The formulae for $\mathcal C^-f$ can be then deduced from $\mathcal C^+ - \mathcal C^- = I$.

This gives a reasonably fast method to compute the coefficients $\eta_{j,n}$, and more importantly, to compute the coefficients of the expansion of $\mathcal C^{\pm}[f_+ + f_-]$ in the basis $R_{j,\alpha}$, in $O(m^2)$ operations.  In practice, one may want to replace $2m+1$ by an even number (or some power of 2) that is larger than $2m+1$ to balance performance the computation of $M_\sigma(\vec k) \vec c$ with that of the FFT that is used in the application of $\mathcal R_{n,0}$.


\bibliographystyle{plain}
\bibliography{references}
\end{document}